\documentclass[a4paper,reqno,oneside]{amsart}

\usepackage{amsmath,amsthm, amssymb,eufrak,yfonts}
\usepackage{geometry} 
\usepackage[utf8]{inputenc} 
\usepackage[T1]{fontenc} 
\usepackage[english]{babel} 
\usepackage{fix-cm}
\usepackage{nicefrac} 
\usepackage{hyperref}
\usepackage{graphicx}
\usepackage{braket}
\usepackage{color}
\usepackage[usenames,dvipsnames]{xcolor}
\usepackage{csquotes}
\usepackage{dsfont}
\usepackage[sc]{mathpazo}
\usepackage{subcaption}
\usepackage{comment}

\newtheorem{proposition}{Proposition}[section]
\newtheorem{theorem}[proposition]{Theorem}
\newtheorem{corollary}[proposition]{Corollary}

\theoremstyle{definition}
\newtheorem{definition}[proposition]{Definition}
\newtheorem{remark}[proposition]{Remark}

\newcommand{\beq}{\begin{equation}}
\newcommand{\eeq}{\end{equation}}
\newcommand{\ben}{\begin{enumerate}}
\newcommand{\een}{\end{enumerate}}
\newcommand{\bit}{\begin{itemize}}
\newcommand{\eit}{\end{itemize}}
\newcommand{\dys}{\displaystyle}

\newcommand{\R}{\mathbb{R}}
\newcommand{\eps}{\varepsilon}
\newcommand{\vfi}{\varphi} 

\providecommand{\abs}[1]{\left|#1\right|}
\providecommand{\norm}[1]{\left \| #1\right \|}
\numberwithin{equation}{section}
\renewcommand{\k}{\mathtt k}

\newcommand{\alphau}{\underline \alpha} 
\newcommand{\alphao}{\overline \alpha} 

\DeclareMathOperator{\od}{\Lambda}
\newcommand{\ind}[1]{\mathds{1}_{#1}}
\newcommand{\sign}{\operatorname{sign}}

\title[Spectral optimization for weighted anisotropic problems with Robin conditions]{Spectral optimization for weighted anisotropic problems with Robin conditions}
\date{}
\author{Benedetta Pellacci}
\address[B. Pellacci]{Dipartimento di Matematica e Fisica,
Universit\`a della Campania  ``Luigi Vanvitelli'',  via A. Lincoln 5, 81100
Caserta, Italy.}
\email[B. Pellacci]{benedetta.pellacci@unicampania.it}
\author{Giovanni Pisante}
\address[G. Pisante]{Dipartimento di Matematica e Fisica,
Universit\`a della Campania  ``Luigi Vanvitelli'',  via A. Lincoln 5, 81100
Caserta, Italy.}
\email[G. Pisante]{giovanni.pisante@unicampania.it}
\author{Delia Schiera}
\address[D. Schiera]{Departamento de Matemática do Instituto Superior Técnico, 
Universidade de Lisboa, 
Av. Rovisco Pais, 
1049-001 Lisboa, Portugal.}
\email[D. Schiera]{delia.schiera@tecnico.ulisboa.pt}

\begin{document}
\maketitle
\begin{abstract}
We study a weighted eigenvalue problem with anisotropic diffusion in bounded Lipschitz domains $\Omega\subset \R^{N}  $, $N\ge1$, under  Robin boundary conditions, proving the existence  of two  positive eigenvalues $\lambda^{\pm}$ 
respectively associated with a positive and a negative eigenfunction.
Next, we analyze the  minimization  of $\lambda^{\pm}$    with respect to the sign-changing weight, showing that the optimal eigenvalues $\Lambda^{\pm}$ are equal and
the optimal weights are of bang-bang type, namely  piece-wise constant functions, each one taking only two values.
As a consequence,  the problem is equivalent to the minimization  with respect to the subsets of $\Omega$ satisfying a volume constraint. Then, we completely solve the optimization problem in one dimension, in the case of  homogeneous Dirichlet or Neumann conditions, showing new phenomena induced by the presence of the anisotropic diffusion.  
The optmization problem for $\lambda^{+}$ naturally arises  in  the study of the optimal spatial arrangement of resources  for a species to survive in an heterogeneous habitat.
\end{abstract}
\noindent
{\footnotesize \textbf{AMS-Subject Classification}}. 
{\footnotesize 49K15, 49K20,  35J92, 35J70.}\\
{\footnotesize \textbf{Keywords}}. 
{\footnotesize Weighted eigenvalues, population dynamics, survival threshold, symmetrization.}

\section{Introduction}

This paper is focused on the spectral optimization problem associated with the following eigenvalue problem
\begin{equation}\label{equazione} \begin{cases}
-{\rm div}\left((H(\nabla u))^{p-1}H_{\xi}(\nabla u)\right) = \lambda m(x) |u|^{p-2}u & \text{ in } \Omega\\
 H^{p-1}(\nabla u) H_\xi(\nabla u)\cdot n +\k|u|^{p-2}u=0 &\text{ on } \partial \Omega,
\end{cases}\end{equation}
where $\Omega\subset \R^N$ is a Lipschitz bounded domain, $N \ge 1$, $\lambda \in \R$, $p >1$,  $n$ is the outward unit normal on $\partial \Omega$ and
$\tau\cdot\eta$  denotes the scalar product between  two vectors $\tau $ and $\eta$. 
The constant $\k$ runs into the set  $[0, +\infty]$, so we are considering homogeneous Robin boundary conditions, which reduce to homogeneous Neumann boundary conditions if $\k=0$, while we will refer to homogeneous Dirichlet case for $\k=+\infty$ (see 
the beginning of Section  \ref{sec:principal eigenvalues} for more details).
 
The function  $m\in L^{\infty}(\Omega)$ is a sign changing  weight, belonging to the class
 \begin{equation}\label{defM}
\mathcal{M}=\left \{ -\beta \leqslant m(x) \leqslant 1, \; \left|\Omega^{+}_{m}\right|> 0,\; \int_{\Omega} m(x)\leqslant -m_0|\Omega|, \, 
 \right \},
 \end{equation}
where $\beta>0$ is  a constant, $m_{0}\in (-1, \beta )$  if $\k>0$,
$m_{0}\in (0,\beta)$  if $\k=0$ and   $\Omega^{+}_{m}:=\{x \in \Omega : m(x)>0 \}$. 
We will assume that the function  $H: \R^N \to \R$, belonging to  $C^2(\R^N \setminus \{ 0\})$,  is such that 
\begin{align}
\label{h:norma}
&\text{$H \geq 0$   and $H(\xi)=0$ if and only if $\xi=0$}
\\
\label{h:posom}
&H(t\xi)= t H(\xi), \quad \text{ for any $t \ge 0$, $\xi\in \R^{N}$}
\\
\label{h:convex}
&\text{$\{\xi \in \R^{N} : H(\xi) < 1\}$ is uniformly convex, }
\end{align}
where by uniform convexity we mean that the principal curvatures of the boundary are positive and bounded away from zero.

We will be interested in the minimization with respect to 
$m\in \mathcal{M}$  of the  positive principal eigenvalues of \eqref{equazione}, namely the positive eigenvalues associated with an eigenfunction of constant sign.

In the case 
\begin{equation}\label{h:modulo}
H(\xi)=|\xi|
\end{equation}
problem \eqref{equazione} corresponds to the linearization  of the  nonlinear elliptic logistic problem
\begin{equation}\label{prob:nonlineariso} 
\begin{cases}
-\Delta u= \lambda  |u|^{p-2}u(m(x)-|u|^{q}) & \text{ in } \Omega\\
\nabla u\cdot n +\k|u|^{p-2}u=0 &\text{ on } \partial \Omega,
\end{cases}
\end{equation}
with $q>0$.  Positive solutions   are the stationary states of the associated reaction diffusion equation. 
This model, introduced in \cite{fis, kpp}, 
describes the dispersal of a population, with density  $u$,  in a heterogeneous 
environment $\Omega$,   triggered by a brownian motion law, so that each individual 
moves in every direction with the same probability.  
The heterogeneity of  the habitat is modelled by representing $\Omega$ as  union of patches, favourable and hostile zones, corresponding respectively  to the positivity and negativity set of the weight
 $m$, so that $\Omega^{+}_{m}$ can be interpreted as the favorable zone of $\Omega$ (see \cite{bhr}).
  
In this context,   a positive principal eigenvalue $\lambda$ with eigenfunction $\vfi$, which, 
in view of  \eqref{h:modulo}, can be chosen positive, turns out to be a  threshold for the 
survival of the population. So that, minimizing $\lambda$, with respect to the weight or to other features  of the model,  endorses the chances of survival.  Several contributions can be found in the literature, and we refer to the recent papers 
\cite{mazari2022,berecoville,dpv, mapeve1, mapeve2, feve} and references therein, for interesting phenomena such as fragmentation effects,  nonlocal aspects or asymptotic analysis. Let us also mention that similar optimization problems have been addressed in other related contexts, such as in the framework of composite membranes (see  \cite{chanillo, henrot} and the references therein). 
The study of the optimization of $\lambda=\lambda(m)$ with respect to the weight 
$m$ goes back to the contribution by Cantrell and Cosner in \cite{CC89} and it 
is known that the minimum $\od$ is achieved by  an optimal weight of bang-bang 
type, namely a  piece-wise function $m=\ind{D}-\beta\ind{D^{c}}$, 
where $\ind{D}$ denotes the characteristic function of the set $D$ and  $D\subset 
\Omega $ turns out to  be a super-level set of the associated positive eigenfunction (see 
\cite{LLNP, DG, CC91,louyan}). 

Then, natural questions  concerning the qualitative properties of the ``optimal set'' $D$ 
arise.
This is a rather hard task,  mostly open in general, and  the analysis is complete only for
$p=2$ and   in dimension one. This situation has been first investigated in \cite{CC91, 
louyan} for homogeneous Dirichlet or Neumann boundary conditions, 
and the study has been concluded in  \cite{LLNP, hintermuller}, where 
 it is proved that   $D$  is connected, so that it is
 an  interval and there exist a constant $\overline{\k}$ such that for every $
 \k>\overline{\k}$,  $D$ is centred  in the middle of the interval $\Omega$, while for  $\k<\overline{\k}$,  $D$ sticks to the boundary.
  For $p\neq2$, the same analysis has been performed  considering  homogeneous Neumann in \cite{DG}.

When the population adopts different diffusion strategies, one is naturally lead to consider different differential operators in the model.  For instance, fractional diffusion operators
have been investigated in \cite{cadiva, dpv, peve} (see also the references therein).
In particular, for spectral fractional laplacian under  homogeneous Neumann boundary conditions  the optimal weight is of bang-bang type (\cite{peve}), while the shape and localization of the optimal set $D$ are still unknown even in dimension one. 

Here, we are focused on anisotropic diffusions, thinking of the population dispersing in the habitat with different probabilities depending on the direction (see\cite{Bouin} for a related model), so that the diffusion operator is given by the so called anisotropic $p-$Laplace operator
\[ 
\Delta_{H, p}u :={\rm div}\left((H(\nabla u))^{p-1}H_{\xi}(\nabla u)\right) .
\]
Eigenvalues' properties when $m\equiv 1$  have been widely studied under various 
boundary conditions, assuming that $H(t\xi)=|t|H(\xi)$  for  every $t\in \R$, in place of 
\eqref{h:posom}
(see e.g. \cite{BFK, dellapietraDBG, GavitoneTrani} and the references therein).
From this perspective, we tackle the case of an indefinite eigenvalue problem under
general Robin boundary conditions. 

As a first result, we establish the existence of a positive principal eigenvalue for every fixed 
$m\in {\mathcal M}$ by minimizing  a suitable Rayleigh quotient restricted to the cone of positive functions (see Proposition \ref{prop:lambda}).

Even in this study a novelty arises: as we want to include the study of an anisotropic 
one-dimensional diffusion operator,  we  assume $H(t\xi)=tH(\xi)$  for  every $\xi\in \R^{N}
$, but just for $t\geq 0$, so that $H$ is not assumed to be a norm as it may not be even. 
 This has significant consequences.
For instance, if one minimizes the associated Rayleigh quotient in the whole Sobolev space,
it is not possible to deduce the sign of the associated  eigenfunction a posteriori (see Remark \ref{rem:sign}).
 This is the reason why we restrict 
the minimization problem in the cone of non-negative functions.

As a matter of fact, there exist two positive eigenvalues $\lambda^{\pm}$
with associated eigenfunctions of constant sign (See Section \ref{sec:principal eigenvalues}). 
One can be obtained through minimization on the cone of the positive functions, the other on the cone of negative ones (see Proposition \ref{prop:exlameno}). This phenomenon, due to the fact that $H$ is not supposed to be even, resembles what occurs  in the context of fully nonlinear operators (see \cite{BD}, \cite{QS} and references therein).

In analogy to what happens for  isotropic diffusions, we 
prove that  $\lambda^{+}(m) $ is a threshold for the existence of positive solutions of the nonlinear logistic elliptic problem (see Theorem \ref{soglia}). 
As a consequence, minimizing $\lambda^{+}(m)$ with respect to $m\in {\mathcal M}$
consists in finding the best spatial arrangements of resources in order to 
endorse the chances of survival of a population living in  $\Omega$.

In this direction we will first prove the following result. 
\begin{theorem}\label{thm:superlevel set}
Assume that $H\in C^2(\R^N \setminus \{ 0\})$ satisfies hypotheses \eqref{h:norma}, \eqref{h:posom} and \eqref{h:convex}. 
The minimization problem
\begin{equation}\label{Lambda}
\od^+:=\inf_{m \in \mathcal{M}} \lambda^+(m) 
\end{equation}
has a solution  given by 
$m(x)=\ind{D}-\beta\ind{D^{c}} $,  $(D^{c}=\Omega\setminus D)$.
If $\varphi=\varphi(m)$ is a positive eigenfunction associated with $\Lambda^+=\lambda^+(m)$, than the set $D$ is a super level set of $\varphi$, i.e. for a suitable $t>0$
\begin{equation}\label{Dset}
 D=\{ x \in \Omega \;:\; \varphi(x) >t\}
\end{equation}
and every  level set of $\varphi$ has zero measure.
\end{theorem}

Theorem \ref{thm:superlevel set} proves that, as shown in several contexts (see for instace \cite{CC91, louyan, LLNP, MazNadPri2020, MazNadPri2023, mazari2023}), also in the anisotropic case 
optimizers for $\Lambda^+$ are of bang-bang 
type.

In particular, the minimization problem \eqref{Lambda} is equivalent to the minimization with respect to the subsets of $\Omega$ satisfying a volume constraint (see Remark \ref{rem:min set}), and one is naturally lead to study qualitative properties of optimal sets. 
With this goal in mind, we consider homogeneous Dirichlet or Neumann boundary conditions and we restrict ourselves to dimension one, where $H$ satisfying \eqref{h:norma}, \eqref{h:posom} and \eqref{h:convex},
has necessarily the expression
\begin{equation}\label{h:1dim}
 H(x)=
\begin{cases}
ax & \text{ if } x \ge 0\\
-bx & \text{ if } x < 0,
\end{cases} 
\end{equation}
with $a\neq b$, otherwise no anisotropy occurs.
We will show the following result.

\begin{theorem}\label{thm:localization}
Let $N=1$, $\Omega=(0,1)$ and assume $H$ is of the form \eqref{h:1dim}. 
\\
Then, the super-level set $D$ is an interval. In addition
\begin{enumerate}
\item If $\k=0$,  $D=(0, |D|)$ if $a>b$, and  $D=(1-|D|, 1)$ if $b >a$.
If  $a=b$,   $(0, |D|)$ and $(1-|D|, 1)$ are both optimal sets. 
\item If $\k=+\infty$, then $D$ is given by 
\begin{equation}\label{localization D} D=\left( \frac{(1-|D|)a}{a+b}, \frac{|D|b+a}{a+b} \right). \end{equation}
\end{enumerate}
\end{theorem}
\begin{figure}
 \centering 
\begin{subfigure}{0.4 \textwidth}
    \centering
  \includegraphics[width=1.1\linewidth]{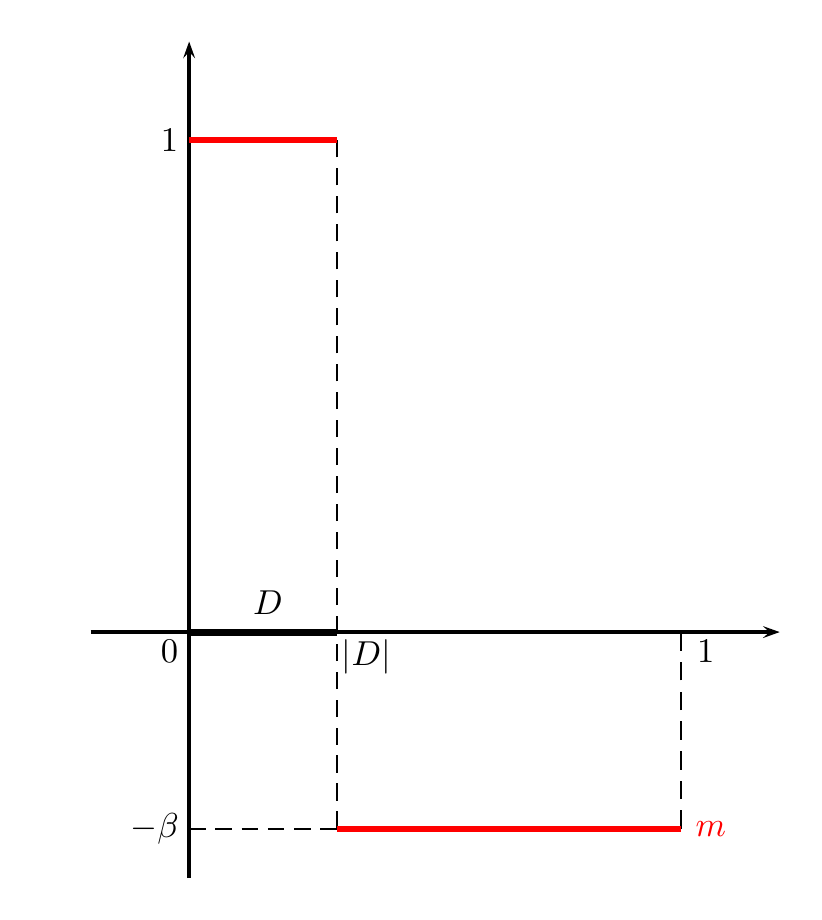}
\caption{$\k=0$ and $a>b$.}
\end{subfigure}%
\hspace{1.3cm}
\begin{subfigure}{0.4 \textwidth}
    \centering
 \includegraphics[width=1.1\linewidth]{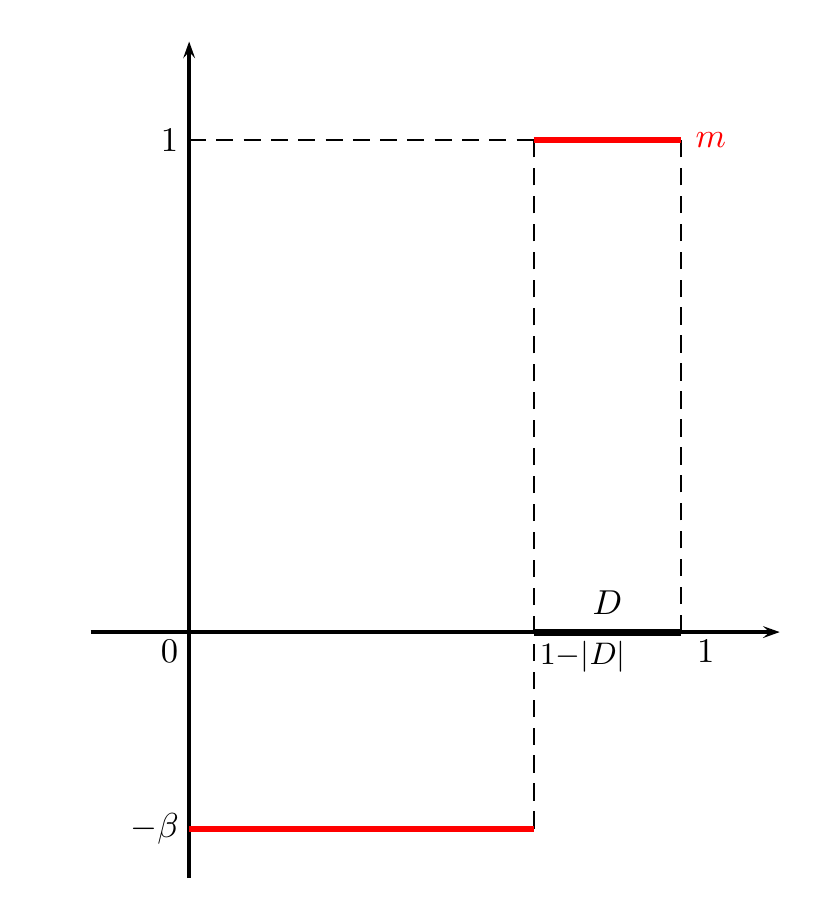}
\caption{$\k=0$ and $a<b$.}
\end{subfigure}
\caption{A representation of the optimal weight $m$ for $\Lambda^+$ in dimension one with Neumann boundary conditions ($\k=0$) (see Theorem \ref{thm:localization}). If $a=b$ the two weights above are both optimal.}\label{fig:loc}
\end{figure}
\begin{figure}
 \centering 
\begin{subfigure}{0.4 \textwidth}
    \centering
  \includegraphics[width=1.1\linewidth]{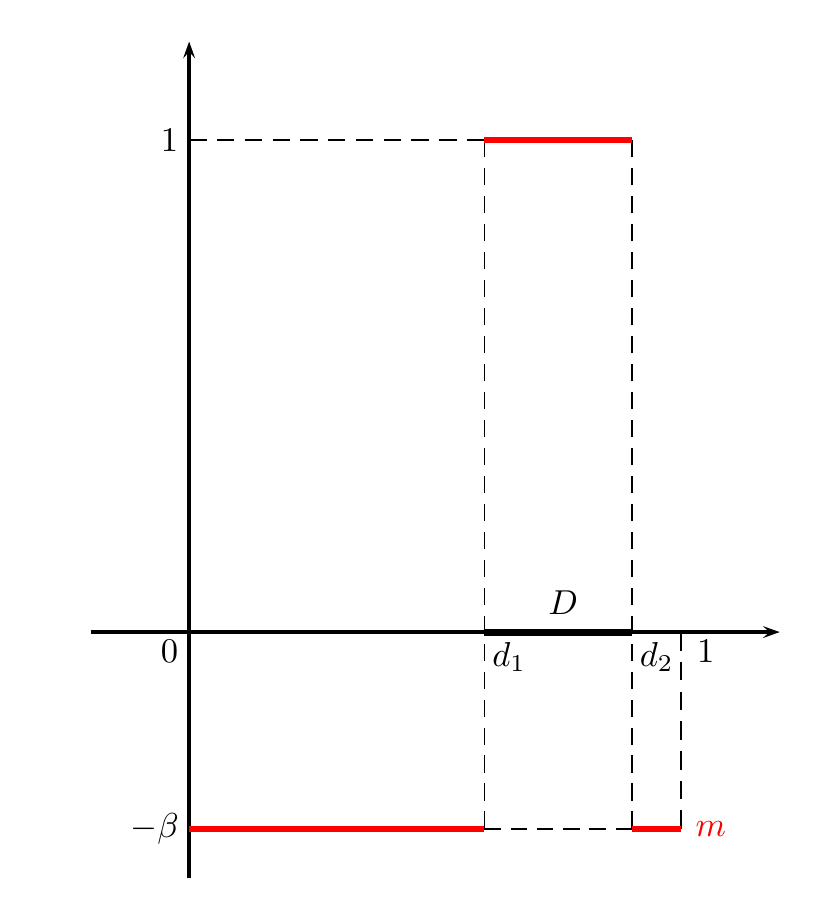}
\caption{Anisotropic case:  $a>b$.}
\end{subfigure}%
\hspace{1cm}
\begin{subfigure}{0.4 \textwidth}
    \centering
 \includegraphics[width=1.1\linewidth]{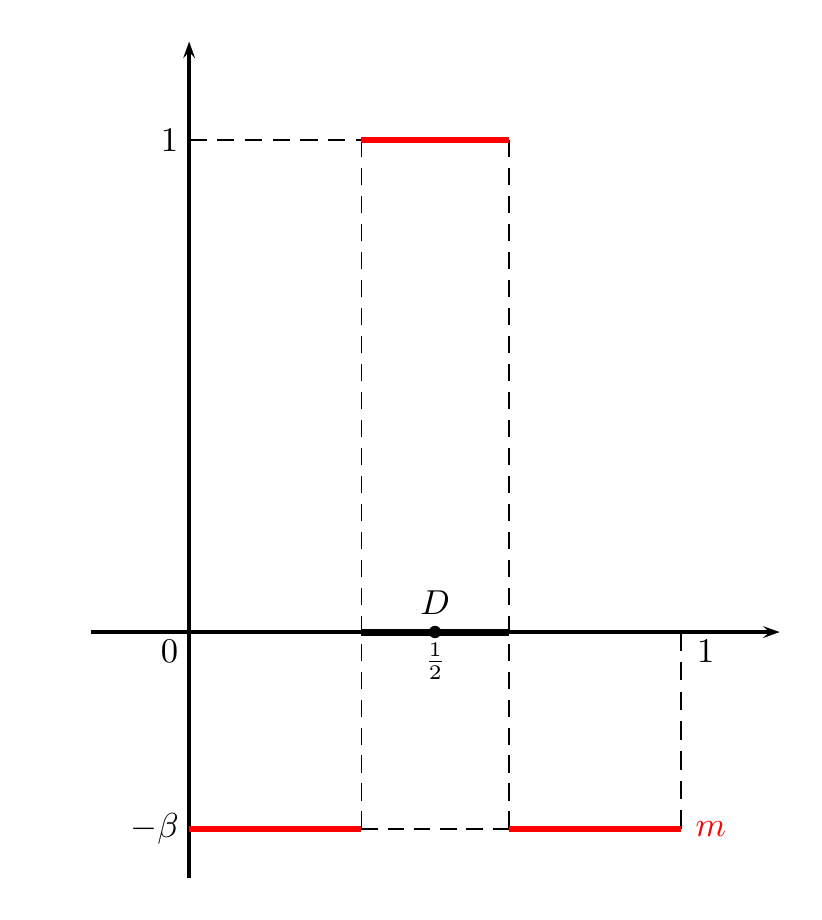}
\caption{Isotropic case: $a=b$.}
\end{subfigure}
\caption{A representation of the optimal weight $m$ for $\Lambda^+$ in dimension one with Dirichlet boundary conditions ($\k=+\infty$), as found in Theorem \ref{thm:localization}. We denote by $D$ the optimal set, $d_1=\frac{(1-|D|) a}{a+b}$, $d_2=\frac{|D|b+a}{a+b}$. In the isotropic case $a=b$ the optimal set is centered in $1/2$. }\label{fig:loc2}
\end{figure}

This result enlighten the effect of the anisotropy on the location of the optimal interval $D$.
Indeed, in the case of homogeneous Dirichlet boundary conditions the anisotropy produces 
a shift of $D$, which turns out to be the one-dimensional anisotropic ball centered at $1/2$, namely
the interval given in \eqref{localization D}. In the case of homogeneous Neumann boundary conditions the anisotropy decides to which extremum the interval $D$ should  stick, see Figures \ref{fig:loc} and \ref{fig:loc2}.

In order to prove Theorem \ref{thm:localization} we will first perform a suitable monotone rearrangement, in the spirit of \cite{LLNP}, to show that the eigenfunction $\vfi$ has 
a unique maximum point, so that
the super-level set  $D$ is an interval. Then, the analysis is completed in the case of homogeneous Neumann and Dirichlet boundary conditions, $\k=0$, $\k=\infty$ respectively.
Let us observe that, when $H(\xi)=|\xi|^{2}$ and Robin boundary conditions are imposed,  the location of $D$  is detected by directly computing and comparing the eigenvalues associated with the possible optimal sets (see \cite{hintermuller}). Here, this comparison appears particularly involved, due to the presence of the anisotropy.
However, when $\k=0$, the monotone rearrangements argument immediately implies that $\vfi$ is monotone in the whole $(0,1)$. Then,  the conclusion follows by direct comparison of the Rayleigh quotient.
In the case of homogeneous Dirichlet boundary conditions,
the uniqueness of the maximum point of the eigenfunction $\vfi$  allows us to manage the equality case in the Polya inequality for anisotropic symmetrizations (see  Proposition \ref{prop: polya anis}) yielding the conclusion.   

We  expect that a suitable version of Theorem \ref{thm:localization}  should also hold in the general case of Robin boundary conditions, as for the isotropic case.

We believe that the case of homogeneous Dirichlet boundary conditions can be handled 
directly exploiting anisotropic symmetrization arguments without passing by monotone rearrangements, and using isoperimetric inequalities, although this would require a careful adaptation in order to handle the case of non-even anisotropy (see Remark \ref{rem:coarea}) . On the other hand, since  monotone rearrangements are used in order to get that $D$ is an interval in the general case of Robin boundary conditions, as a byproduct we have that the eigenfunction $\vfi$ has a unique maximum point.
This allows us to use an elementary  one-dimensional rigidity result for the equality case in the Polya inequality (see Proposition \ref{prop: polya anis}), to obtain the localization of $D$.

Starting the analysis in the cone of negative functions, it is possible 
to show the existence of a positive eigenvalue $\lambda^-(m)$, with a unique negative
normalized eigenfunction. Then, one can perform the optimization of $\lambda^-(m)$
with respect to the weight $m$, namely studying  the problem 
\[ 
\od^-:=\inf_{m \in \mathcal{M}} \lambda^-(m).
\]
It is possible to give the analogous version of Theorem \ref{thm:localization} for the optimal set associated with $\Lambda^{-}$.
Also, we show that under a suitable symmetry assumption on the domain $\Omega$ one has 
the following result
\begin{theorem}\label{thm:Lambda+=Lambda-}
We assume that $\Omega$ has a centre of symmetry, i.e. there exists $x_{0}\in \R^{N}$ such that $2x_{0}-x \in\Omega$ for every $x\in \Omega$.  Then, 
\[ 
\Lambda^+=\Lambda^-. 
\]
\end{theorem}

We stress that our proof of Theorem \ref{thm:Lambda+=Lambda-} heavily uses the monotonicity properties of the eigenfunctions associated  with the optimal eigenvalues. We think that the equality does not hold for any $\lambda^{\pm}(m)$ as the monotonicity properties are not expected to hold.

The paper is organized as follows. In the next section we discuss the emergence of $\lambda^{+}(m)$ as a threshold for the existence of positive solution of the associated nonlinear elliptic problem. In Section \ref{sec:principal eigenvalues} we set up the eigenvalue problem for a fixed weight $m$, we show the existence of the principal eigenvalues  $\lambda^{+}(m)$ and $\lambda^{-}(m)$ in Propositions \ref{prop:lambda} and \ref{prop:exlameno},  then we give the proof of Theorem \ref{thm:superlevel set} as well as the counterpart for $\lambda^{-}(m)$. The proofs of Theorems \ref{thm:localization} and \ref{thm:Lambda+=Lambda-}  are given in Section \ref{localization}. Finally in Section \ref{rearrangements} we prove some necessary technical tools about one-dimensional anisotropic rearrangements inequalities.

\section{The nonlinear Problem }\label{nonlinear}

In the whole paper we will assume that $H$ satisfies \eqref{h:norma}, \eqref{h:posom}. This type of functions is usually referred to as {\it positively homogeneous Minkowski norm}. It can be easily seen that $H$ satisfies the following growth conditions
\begin{equation}\label{h:growth}
\alphau \abs{\xi} \leqslant H(\xi) \leqslant \alphao \abs{\xi}, \qquad \text{with $0<\alphau<\alphao$.}
\end{equation}
Under the assumption \eqref{h:convex}, it can also be proved that for any $p>1$ the function $H^p$ is strictly convex and there exist positive constants $\gamma, \Gamma$ such that 
\begin{equation}
\label{H-convexity}
 (\text{Hess}(H^p)(\xi))_{ij} \zeta_i \zeta_j \ge \gamma  |\xi|^{p-2} |\zeta|^2,
\qquad 
\sum_{i, j=1}^N \abs{(\text{Hess}(H^p)(\xi))_{ij}} \leqslant \Gamma |\xi|^{p-2} 
\end{equation}
for any $\xi \in \R^N \setminus\{0\}$ and $\zeta\in \R^N$, see \cite[Proposition 3.1]{CFV14_1} for dimension $N \ge 2$, whereas for $N=1$ these estimates follow by direct computations.

Let us start by giving the following definition. 

\begin{definition}\label{def:lambda}
Let $m\in {\mathcal M}$. We set 
\[
\lambda^+(m)=\inf_{u\in {\mathcal S_{\k, m}^+}}{\mathcal R_{\k,m}}(u)
\]
where the Rayleigh  quotient ${\mathcal R_{\k,m}}$  and the set 
${\mathcal S_{\k, m}^+}$
are defined  in a different way depending on $\k$.
If $\k<+\infty $ they are given by
\[
{\mathcal R_{\k,m} }(u)
:=
\dfrac{\int_\Omega H(\nabla u)^pdx+\k\int_{\partial \Omega}|u|^{p}d\sigma}{\int_\Omega m(x)|u|^{p}dx}
\] 
\[
{\mathcal S_{\k, m}^+} 
:=\left\{u\in W^{1, p}(\Omega),\ u \ge 0, \, \int_\Omega m(x)|u|^p\,dx>0\right\}, 
\]
and for $\k=+\infty$, 
\[
{\mathcal R_{\infty,m}}(u):=
\dfrac{\int_\Omega H(\nabla u)^p}{\int_\Omega m(x)|u|^{p}dx}
\qquad
{\mathcal S_{\infty, m}^+}:=\left\{  u\in W^{1, p}_{0}(\Omega),\ u \ge 0, \, \int_\Omega m(x)|u|^p\,dx>0\right\}.
\]
\end{definition}

A similar definition can be given choosing ${\mathcal S}^{-}_{\k,m}$ (see Definition \ref{def:lambda-}).

In the next theorem we show that, as in the isotropic case, $\lambda^{+}(m)$ naturally arises as a threshold for the existence of positive solutions of the following logistic type nonlinear problem 
\begin{equation}\label{eq:nonlinear} \begin{cases} 
- \Delta_{H, p}u = \lambda |u|^{p-2} u (m - |u|^q) &\text{ in } \Omega,\\
 H(\nabla u)^{p-1} H_\xi(\nabla u)\cdot n +\k|u|^{p-2}u=0 &\text{ on } \partial \Omega,
\end{cases}\end{equation}
with $q>0$, that is the counterpart in the anisotropic case of problem \eqref{prob:nonlineariso}. 

\begin{theorem}\label{soglia}
There exists a  nonnegative nontrivial bounded solution $u$
to problem \eqref{eq:nonlinear}  if and only if $\lambda > \lambda^+(m)$. Moreover, 
$u\in C^{1, \alpha}(\Omega)$ is the  unique nonnegative nontrivial solution. 
\end{theorem}

This result can be proved by means of different approaches, such as minimizing a suitable action functional. Here, we will obtain the existence of a bounded solution via sub-, super-solution arguments, which can be also exploited to show the existence of solutions of the associated parabolic equations.

We recall that $u\in W^{1,p}(\Omega)$ is a super-solution for \eqref{eq:nonlinear}
if 
\begin{equation}\label{eq:supersolution} 
\begin{split}
D(u,\phi): = & \int_\Omega H(\nabla u)^{p-1} H_\xi(\nabla u) \nabla \phi - \lambda \int_\Omega m(x) |u|^{p-2} u \phi  \\ 
& + \lambda\int_\Omega  |u|^{p+q-2} u \phi + \k \int_{\partial \Omega} |u|^{p-2} u \phi \geqslant 0. 
\end{split}
\end{equation}
is satisfied for any nonnegative $\phi \in C^{\infty}(\Omega)$ if   $\k\in (0,\infty)$ or $\phi \in C_{c}^{\infty}(\Omega)$ if $\k = \infty$. Analogously, but with the opposite inequality, one can give the definition of a sub-solution for \eqref{eq:nonlinear}. We explicitly note that if $u\in L^{(p^{*})'/(p+q-1)}(\Omega)$ then $D(u,\cdot)$ is sequentially continuous with respect to $\| \cdot \|_{W^{1,p}(\Omega)}$, therefore, for bounded super-solutions, \eqref{eq:supersolution} holds for any $\phi \in W^{1,p}(\Omega)$ if $\k\in (0,\infty)$ or $\phi \in W^{1,p}_{0}(\Omega)$ if $\k = \infty$.

We start by proving an existence result. 

\begin{proposition}\label{prop:subsuper}
Suppose that there exist a sub-solution $\underline u$ and a super-solution $\overline u$ to \eqref{eq:nonlinear}, and assume that for some constants $\underline c$, $\overline c$ one has
$ -\infty < \underline c \leqslant \underline u \leqslant \overline u \leqslant \overline c < \infty$ almost everywhere in $\Omega$. Then there exists a weak solution $u$ of \eqref{eq:nonlinear} such that  $\underline u \leqslant u\leqslant \overline u$ a.e. in $\Omega$.
\end{proposition}
\begin{proof}
The proof follows the classical argument  in \cite[Theorem 2.4]{Struwe}.  
We will only highlight some different  points  which appear when considering Robin boundary conditions, i.e. $\k <\infty$. The approach is based on the minimization
of the functional
\[ 
E(u)= \int_\Omega H(\nabla u)^p + \k \int_{\partial \Omega} |u|^p- \lambda \int_\Omega m |u|^p + \lambda \int_\Omega |u|^{p+q}
 \]
on the closed convex subset of $W^{1, p}(\Omega)$
\[ 
\mathcal{M}=\left\{ u \in W^{1, p}(\Omega): \underline u \leqslant u \leqslant \overline u \, \text{ a.e.} \right\}.
\]

We first observe that, being $\mathcal{M}$ bounded in $L^{\infty}(\Omega)$, using the hypotheses on $H$, it is not hard to verify that the functional $E$ restricted to $\mathcal{M}$, endowed with the norm induced by the  $W^{1, p}(\Omega)$ norm, is coercive and weakly lower semicontinuous. It easily follows that $E$ attains its infimum in $\mathcal M$. Let $u\in \mathcal{M}$ be the minimum point of $E$ on $\mathcal{M}$. 
Since $p>1$ implies that $H^{p}$ is positively $p$-homogeneous (therefore with first derivatives positively $(p-1)$-homogeneous and continuous in the origin), the Lagrangian function of the bulk component of the functional $E$, namely
\[
f(x,u,\xi)=H^{p}(\xi)-\lambda m(x) |u|^p+\lambda |u|^{p+q},
\]
satisfies the so called {\it natural growth conditions} (cf. \cite[Condition 3.34]{daco-2008}). Then one may argue as in \cite[Theorem 3.7]{daco-2008}, to prove that the right-Gateaux derivative of the functional $E$ at the minimum point $u$ in any direction $\phi=v-u$ with $v\in \mathcal{M}$\footnote{Note that, since $\mathcal{M}$ in convex, $u+\eps(v-u)$ is an admissible variation for the functional $E$ restricted to $\mathcal{M}$. Again for the lower order terms, as well as for the boundary integral, we are heavily using the fact that $\mathcal{M}$ is bounded in $L^{\infty}(\Omega)$ to apply the dominated convergence theorem.} exists and it is equal to $D(u,\phi)$. Moreover, being $u$ a minimizer we also have $D(u,\phi)\geqslant 0$.

Given $\varphi \in C^\infty( \Omega)$, $\eps>0$ sufficiently small, and having defined 
\[ 
\varphi^\eps:= \max \{ 0, u+ \eps \varphi - \overline u \} \ge 0,
\quad \varphi_\eps:= \max \{ 0, \underline u - (u+\eps \varphi) \} \ge 0, 
\]
we have that $v_\eps:=u + \eps \varphi - \varphi^\eps + \varphi_\eps \in \mathcal{M}$, therefore the variational inequality  
$D(u,v_{\eps}-u) \geqslant 0$ holds true. By linearity it results 
\[
D(u,\varphi) \geqslant \frac{1}{\eps} \big(D(u,\varphi^\eps)- D(u,\varphi_\eps)  \big).
\]
Let us show that the right hand side goes to zero when $\eps$ vanishes.

In the following, we will use the notation
\begin{align*}
g(x, u)&=\lambda m(x) |u|^{p-2} u  - \lambda |u|^{p+q-2} u, \quad
 \Omega_\eps:= \left\{ x \in \Omega: u(x)+\eps \varphi(x) \ge \overline u(x) > u(x) \right\}. 
\\
K_\eps &=\left\{ x\in \partial \Omega: Tr \varphi^\eps(x) \ne 0\right\}, \quad \text{ where }  Tr: W^{1, p}(\Omega) \rightarrow L^p(\partial \Omega) \text{ is the trace operator.}
\end{align*}
Taking into account that $\overline u$ is a super-solution, that $u+\eps\vfi-\overline{u}\leqslant \eps \vfi$ in $\Omega_{\eps}$, and by convexity of $H(\cdot)^{p}$ and $|\cdot|^{p}$, one has
\[
\begin{split}
D(u,\varphi^\eps)  \geqslant  &  D(u,\varphi^\eps) - D(\overline u,\varphi^\eps)  \\
\geqslant & \; \eps \int_{\Omega_\eps} \left( H(\nabla u)^{p-1} H_\xi(\nabla u) - H(\nabla \overline u)^{p-1} H_\xi(\nabla \overline u) \right) \nabla \varphi \\
& + \eps \, \k \int_{K_\eps} (|\overline u|^{p-2} \overline u-  |u|^{p-2} u) \varphi -  \eps \int_{\Omega_\eps} \abs{g(x, u)- g(x, \overline u)} |\varphi|.
\end{split}
\]

Note that $\vfi^{\eps}\to 0$ in $W^{1,p}(\Omega)$, so that
$Tr\vfi^{\eps}\to 0$ in $L^{p}(\partial\Omega)$,  thanks to  the continuity of the trace operator; as a consequence, $|\Omega_\eps| \to 0$, and $\mathcal{H}^{n-1}(K_\eps) \to 0$. Then
\[ 
D(u,\varphi^\eps) \geqslant o(\eps). 
\]
and the conclusion follows arguing as in the proof of \cite[Theorem 2.4]{Struwe}.
\end{proof}

In view of Proposition \ref{prop:subsuper}, in order to prove the existence of a bounded weak solution to \eqref{eq:nonlinear}, it is sufficient to find $\underline{u}\leqslant \overline{u}$ bounded sub- and super-solution.
To this aim, it is convenient to introduce the following eigenvalue for any $\lambda>0$ fixed
\[ 
\mu^+(\lambda, m)
=\inf_{\substack{v \in W^{1, p} \\ v \ge 0, \not \equiv 0}} \frac{\int_\Omega H(\nabla v)^p + \k \int_{\partial \Omega} |v|^p - \lambda\int_\Omega m |v|^p}{\int_\Omega |v|^p}.
 \]
Arguing as in Proposition \ref{prop:lambda} it is possible to show that $\mu^+(\lambda, m)$ is attained by a positive eigenfunction which we denote by $\Phi\in {\mathcal S}^{+}_{\k,m}$. 
The eigenvalue $\mu^+(\lambda, m)$ is the  anisotropic counterpart of the eigenvalue introduced in \cite{bhr,rh} and as in the isotropic diffusion we are going to show that 
$\mu^+(\lambda, m)<0$ is  a sufficient and necessary condition to obtain the existence of
 a positive solution of \eqref{eq:nonlinear}.

\begin{proof}[Proof of Theorem \ref{soglia} ]
We only consider the case $\k< \infty$, the case of  Dirichlet boundary conditions ($\k=\infty$) can be treated analogously.

\textit{First step: existence of a non-negative nontrivial solution above $\lambda^+(m)$.} 
Let us fix $\lambda > \lambda^+(m)$.  Note that $\mu^+(\lambda, m)<0$, since for any $v \in {\mathcal S}^{+}_{\k,m}$ we have
\begin{align*}  
\dys \dfrac{\int_\Omega H^p(\nabla v) + \k \int_{\partial \Omega} |v|^p - \lambda\int_\Omega m |v|^p}{\int_\Omega |v|^p} \frac{\int_\Omega |v|^p}{\int_\Omega m |v|^p} &= \frac{\int_\Omega H^p(\nabla v) + \k \int_{\partial \Omega} |v|^p - \lambda\int_\Omega m |v|^p}{\int_\Omega m |v|^p} 
\\
&= \frac{\int_\Omega H^p(\nabla v) + \k \int_{\partial \Omega} |v|^p}{\int_\Omega m |v|^p} - \lambda.
 \end{align*}
Having fixed $\eps >0$ such that 
\[ 
\eps < \min\left\{\left( - \frac{\mu^+(\lambda, m)}{\lambda} \right)^{\frac 1 q}, \|m^{+}\|_{\infty}\right\},
\]
let $\Phi$ be the positive eigenfunction associated with $\mu^+(\lambda, m)$ satisfying $\|\Phi\|_{\infty}=1$. It is immediate to see that  $\underline u:=\eps \Phi >0$ is a sub-solution; in addition any constant function $\overline u\equiv c \geqslant (\norm{m^{+}}_\infty)^{\frac 1 q}$ is a super-solution, and $\underline u<\overline u$ thanks to the choice of $\eps$.
Then,  Proposition \ref{prop:subsuper} yields the existence of  a solution $u$ such that $0 <\underline u \leqslant u \leqslant \overline u$. As a consequence, $0<u \leqslant \|m^{+}\|^{\frac 1 q}_{\infty}$.

\textit{Second step: non-existence of a nontrivial nonnegative solution below $\lambda^+(m)$.} 
Assume $0 < \lambda \leqslant \lambda^+(m)$, and let, by contradiction, $u$ be a nonnegative,
 nontrivial, bounded solution to \eqref{eq:nonlinear}.  Then, by testing \eqref{eq:nonlinear} with $u$, 
one immediately see that $\int_\Omega m u^p >0$. So that Definition \ref{def:lambda} yields
\begin{align*} 
\lambda^+(m) \int_\Omega m u^p & \leqslant \int_\Omega H(\nabla u)^p + \k \int_{\partial \Omega} u^p = \lambda \int_\Omega m u^{p} -\lambda \int_\Omega  u^{p+q}
< \lambda^+(m) \int_\Omega m u^{p}, 
\end{align*}
which is a contradiction. 

\textit{Third step: uniqueness and regularity.} 
Classical elliptic regularity theory (see for instance \cite{Tolksdorf}) and the Harnack inequality proved in \cite{Trudinger.1967} (see also \cite{dellapietragavitone})  ensure that $u\in C^{1,\alpha}(\Omega)$ and positive. The uniqueness can be obtained arguing by contradiction. Assume that $u, v$ are two  positive solutions to \eqref{eq:nonlinear} for $\lambda>\lambda^{+}(m)>0$. Let $\eps>0$ and taking $ \frac{u^p}{(v+\eps)^{p-1} }$ as test function in the equation satisfied by $v$. By applying a suitable Picone identity (\cite[Lemma 2.2]{Jaros}), one has

\[
\begin{split}
\k \int_{\partial \Omega} u^{p} & \left( 1- \frac{v^{p-1}  }{(v+\eps)^{p-1}} \right) - \lambda  \int_{\Omega} m u^{p} \left( 1- \frac{v^{p-1}}{(v+\eps)^{p-1}} \right) + \lambda  \int_\Omega  u^{p }\left( u^{q}- \frac{v^{p+q-1}}{(v+\eps)^{p-1}} \right) 
\\
= 
&  \int_\Omega H(\nabla v)^{p-1} H_\xi (\nabla v) \nabla \left ( \frac{u^p}{(v+\eps)^{p-1}} \right)-  \int_\Omega H^p(\nabla u) \leqslant 0.
\end{split}
\]
Writing the analogous inequality obtained taking $ \frac{v^p}{(u+\eps)^{p-1} } $ as a test function in the equation satisfied by $u$, summing up and  and letting $\eps \to 0$, we obtain
\begin{align*} 0  & \ge \lambda \int_\Omega (u^p-v^p)(u^q-v^q)  \\
&=\lambda \left( \int_{\Omega \cap \{ u \ge v\}} (u^p-v^p)(u^q-v^q) + \int_{\Omega \cap \{ u \leqslant v\}}(u^p-v^p)(u^q-v^q) \right) \ge 0, \end{align*}
a contradiction. 
\end{proof}

\section{The principal  eigenvalues $\lambda^\pm(m)$}\label{sec:principal eigenvalues}

We start this section  focusing our attention on the properties of $\lambda^+(m)$ (cf. Definition \ref{def:lambda}).
Let us first observe that,
an easy adaptation of  Proposition 2 in \cite{LLNP} allows us to
refer to the case $\k=+\infty$ as the case  in which homogeneous Dirichlet boundary conditions are imposed.

The existence of $\lambda^{+}(m)$ (see Definition \ref{def:lambda}) will be proved  via a constrained minimization
  as shown in the next proposition. 

\begin{proposition}\label{prop:lambda}
Assume \eqref{h:norma}, \eqref{h:posom}, \eqref{h:convex}. Let $m\in {\mathcal M}
$. Then, the following conclusions hold. 
\begin{enumerate}
\item 
$\lambda^{+}(m)$ is attained by a unique positive  $\vfi\in C^{1,\alpha}(\Omega)$ up to multiplication by a positive constant.
\item
 $\lambda^+(m)$ satisfies the lower bound,
\[
\lambda^+(m)\geq \frac{\alphau c}{\max\{1,\beta\}}
\]
where $c=c(\Omega)>0$ is a positive constant, $\beta$ is given  in \eqref{defM}, and $\underline \alpha$ is given in \eqref{h:growth}.
\item 
$\lambda^{+}(m) $ is the unique positive principal eigenvalue with positive eigenfuncion.
\end{enumerate}
\end{proposition}

\begin{remark}
Let us observe that in the case $H(\xi)=|\xi|$, $\k>0$, if the condition $\int_{\Omega} m(x)|u|^p>0$ is not assumed,  there exist two principal eigenvalues $\lambda_-<0<\lambda_+$ with associated positive eigenfunctions $\varphi_-,\,\varphi_+$ such that respectively $\int_{\Omega} m(x)|\varphi_{\pm}|^p\gtrless 0$. Whereas, for $\k=0$, $\lambda_-=0$ is a principal eigenvalue associated with a constant eigenfunction and $\lambda_+>0$ iff $\int m(x)<0$. In this case $\int_{\Omega} m(x)|\varphi_+|^p>0$. Here, we focus on positive principal eigenvalues, this is why we impose  $\int_{\Omega} m(x)|u|^p>0$.

\end{remark}
\begin{remark}
Notice  that, if $\Omega$ is a $C^{1, \alpha}$ domain, then the eigenfunction $\vfi \in C^{1, \alpha}(\overline \Omega)$ by \cite[Theorem 2]{Lieberman}. 
\end{remark}

\begin{proof}[Proof of Proposition \ref{prop:lambda}] 
Let us start by considering the case $\k<+\infty$.  The invariance of ${\mathcal R_{\k,m}}(u)$ up to a positive scaling of $u$ ensures that we can calculate $\lambda^{+}(m)$ by solving a constrained minimization problem, i.e.   
\begin{equation}\label{minpb}
\lambda_{\k}^{+}(m)=\inf \left \{ \int_{\Omega} H^p(\nabla v) + \k \int_{\partial \Omega} |v|^p : v \in W^{1, p}(\Omega), \, v \ge 0, \, \int_{\Omega} m(x) \abs{v}^p =1 \right \},
\end{equation}
and note that $\lambda^{+}(m)=\lambda_{\k}^{+}(m)$.
Notice that the set where we minimize is not empty: indeed, it  suffices to take $\psi \in C^{\infty}_{c}(\Omega)$,  $\psi \ge 0$, approximating  the characteristic function of the set $\Omega^{+}_{m}$
in $L^{p}(\Omega)$, and then normalize it, in order to have a function  satisfying the constraint. 

As a consequence, $\lambda^+(m)$ is finite and the direct methods of calculus of variations easily imply the existence of a minimizer. For it, let us take a minimizing sequence $\{u_n\}$ with energy bounded by a positive constant $C$, so that $u_{n}$ satisfies
\[ 
\int_{\Omega} (H(\nabla u_n))^p + \k \int_{\partial \Omega} |u_n|^p  \leqslant C.
 \]
First we observe that there exists a constant $c>0$ such that for any $m\in \mathcal M$ and for any $u\in W^{1,p}$ with $\int_{\Omega} m(x) \abs{v}^p  >0$ the following Poincar\'e type inequality holds
\begin{equation}\label{poincare-type}
\int_\Omega |u|^p \leqslant c \left(  \int_{\Omega} (H(\nabla u))^p + \k \int_{\partial \Omega} |u|^p   
\right).
\end{equation} 
Indeed, equation \eqref{poincare-type} reduces to  \cite[Lemma 3.1]{DG}  when $\k=0$ and can 
be easily proved when $\k>0$ arguing by contradiction. 
Using \eqref{poincare-type} we deduce that the sequence $\{u_n\}$ is bounded in $W^{1, p}(\Omega)$. Thus,  there exists $u\in W^{1, p}(\Omega)$, $u\geq 0$, such that, up to a subsequence, $u_n \rightharpoonup u$ in $W^{1, p}(\Omega)$ and $u_{n}\to u$ strongly in $L^{p}(\Omega)$, so that $u$ satisfies the
 constraints.   Finally, taking into account that $H$ is continuous and convex and that  the embedding  $W^{1, p}(\Omega) \hookrightarrow L^p(\partial \Omega)$ is compact, one obtains that $\lambda_{\k}^+(m)$ is attained by $u$ since
\[ \lambda_{\k}^+(m) \leqslant \int_{\Omega}(H(\nabla u))^p+\k \int_{\partial \Omega} |u|^p \leqslant \liminf_{n\to \infty} \left( \int_{\Omega}(H(\nabla u_n))^p+\k \int_{\partial \Omega} |u_n|^p \right)=\lambda_\k^+(m). 
\]
Being $u$ a minimizer for the problem \eqref{minpb}, it is a weak solution of the associate Euler-Lagrange equation \eqref{equazione}. We can therefore apply the classical elliptic regularity theory (see for instance \cite{Tolksdorf}) and the Harnack inequality proved in \cite{Trudinger.1967} (see also \cite{dellapietragavitone})  to ensure that $u\in C^{1,\alpha}(\Omega)$ and positive.

The Dirichlet case, corresponding to $\k=+\infty$, can be addressed in an analogous way, by observing that,
\begin{equation}\label{minpbdir}
\lambda_\infty^+ (m)=\inf \left \{ \int_{\Omega} (H(\nabla v))^p   : v \in W_{0}^{1, p}(\Omega), \, v \geqslant 0, \,  \int_{\Omega} m(x) \abs{v}^p =1 \right \}. 
\end{equation}
and the existence of a minimum point follows as before by exploiting
the classical Poincar\'e inequality.

We now prove the uniqueness of the minimizer, arguing by contradiction and exploiting a convexity argument as done in \cite{BFK}. Take two positive minimizers $u$ and $U$. For $t \in (0, 1)$ we set 
 $\eta=tu^p + (1-t)U^p$, and $u_t=\eta^{1/p}$. 
Then, \eqref{h:posom},  \eqref{h:convex} and the convexity of $t \mapsto t^{p}$ imply that
\begin{align}
\nonumber H^p(\nabla u_t) &= \eta H^p \left( \frac{t u^p }{\eta} \frac{ \nabla u}{u} + \frac{(1-t)U^p }{\eta} \frac{\nabla U}{U} \right)
 \\
\label{inequality:uniq1} & \leqslant
 \eta \left[ \frac{tu^p}{\eta} H\left( \frac{\nabla u}{u} \right)
  +
 (1-t)\frac{U^p}{\eta}  H \left( \frac{\nabla U}{U} \right) \right]^p
 \\
 \label{inequality:uniq2} & \leqslant
 \eta \left[ \frac{tu^p}{\eta} H^p\left( \frac{\nabla u}{u} \right)
  +
 (1-t)\frac{U^p}{\eta}  H^p \left( \frac{\nabla U}{U} \right) \right]
 \\
\nonumber& = t H^p(\nabla u) + (1-t)H^p(\nabla U). 
\end{align}
Furthermore
\[
 \int_{\partial \Omega} u_t^p= t \int_{\partial \Omega} u^p + (1-t) \int_{\partial \Omega} U^p ,\qquad \int_{\Omega}m(x)u_{t}^{p}=1.
 \]
The above expressions imply that also $u_t$ is a minimizer, which means that 
\eqref{inequality:uniq1} and \eqref{inequality:uniq2} are actually equalities. 
 Then, as $p>1$, $H^{p}$ is strictly convex and one obtains  that  $\nabla u/u=\nabla U/U$. This 
 immediately  yields that
$u/U$ is constant almost everywhere, and  in view of the constraint condition in \eqref{minpb}, we deduce that $u=U$.  This finishes the proof of the first conclusion.

In order to obtain conclusion (2), one may use  again the Poincaré type inequality \eqref{poincare-type} or the classical Poincaré inequality, and taking into account \eqref{h:growth}, \eqref{defM},
 one obtains
\[
1=\int_{\Omega}m(x)|u|^{p}\leqslant \|m\|_{\infty}\|u\|_{p}^{p}\leqslant 
\frac{\|m\|_{\infty}}{\alphau c}\lambda^+(m)
\leqslant \frac{\max\{1,\beta\}}{\alphau c}\lambda^+(m).
\]

To conclude the proof it is left to show that,  if there exists an eigenvalue $\lambda>0$ such that the corresponding eigenfunction is non-negative, then  $\lambda=\lambda^+(m)$. 
 Let us assume that $v \geqslant 0$ is an eigenfunction with eigenvalue $\lambda \ne \lambda^+(m)$, $\lambda>0$, and let us first consider the case $\k < +\infty$. 
  Take $u=t \bar u$, where $t >0$, and $\bar u$ is the positive eigenfunction for $\lambda^+(m)$ normalized such that $\int_\Omega m \bar u^p=1$. 
Notice that by \eqref{h:posom} $u$ is again an eigenfunction for \eqref{equazione} with eigenvalue $\lambda^+(m)$. 

Taking as test function $ u^p(v+\eps)^{1-p} $
in the equation satisfied by $v$,  one obtains
\begin{align*} 
\int_\Omega H(\nabla v)^{p-1} H_\xi (\nabla v) \nabla \left ( \frac{u^p}{(v+\eps)^{p-1}} \right)  &+ \k \int_{\partial \Omega} \frac{u^pv^{p-1}  }{(v+\eps)^{p-1}} 
=
 \lambda \int m \frac{u^pv^{p-1} }{(v+\eps)^{p-1}} 
\\
=&  \lambda \int m \frac{u^pv^{p-1}}{(v+\eps)^{p-1}}   - \lambda^+(m) \int_\Omega m u^p 
\\
&+ \int_\Omega H^p(\nabla u) + \k \int_{\partial \Omega} u^p .
\end{align*}
Choosing as test function $ v^p (u+\eps)^{1-p}$ in the equation satisfied by $u$ yields
\begin{align*}
\int_\Omega H(\nabla u)^{p-1} H_\xi (\nabla u) \nabla \left ( \frac{v^p}{(u+\eps)^{p-1}} \right)  &+ \k \int_{\partial \Omega}\frac{v^p u^{p-1}}{(u+\eps)^{p-1}} =
 \lambda^+(m) \int_\Omega m \frac{v^p u^{p-1}}{(u+\eps)^{p-1}} \\
 & -\lambda \int_\Omega m v^p + \int_\Omega H^p(\nabla v) + \k \int_{\partial \Omega} v^p. 
\end{align*}
Summing up these two identities and using Picone identity \cite{Jaros}, we deduce that
\begin{align*} 
\k \int_{\partial \Omega} \frac{u^pv^{p-1}  }{(v+\eps)^{p-1}}  + \k \int_{\partial \Omega}\frac{v^p u^{p-1}}{(u+\eps)^{p-1}}
 \geqslant
 &\lambda \int m \frac{u^pv^{p-1}}{(v+\eps)^{p-1}}  - \lambda^+(m) \int_\Omega m u^p \\
&  \lambda^+(m) \int_\Omega m \frac{v^p u^{p-1}}{(u+\eps)^{p-1}}  -\lambda \int_\Omega m v^p  \\
&+ \k \int_{\partial \Omega} u^p + \k \int_{\partial \Omega} v^p.
\end{align*}
Letting $\eps \to 0$, one gets 
\[ (\lambda-\lambda^+(m)) \int_\Omega m (u^p-v^p) \leqslant 0. \]
Since $\lambda > \lambda^+(m)$ as $v\in {\mathcal S^{+}_{\k,m}}$, we conclude
\[ t^p=\int_\Omega m u^p \leqslant \int_\Omega m v^p. \]
Since $t>0$ is arbitrary, we get a contradiction, hence $\lambda=\lambda^+(m)$.
\end{proof}

\begin{remark}\label{rem:sign}
Assuming $H$ to be even, it results
\[
\lambda^{+}(m)=\lambda(m)=\min\left\{\int_{\Omega}|H(\nabla u)|^{p}+\k\int_{\partial \Omega}|u|^{p}, \;u\in W^{1,p}(\Omega)\; :\; \int_{\Omega}m(x)|u|^{p}=1 
\right\}.
\]
Indeed, arguing as in Proposition \ref{prop:lambda} one can show that $\lambda(m)$ is achieved by 
a function $u$. In order to show that $u$ has constant sign, we can argue by
contradiction, as for example in \cite[Theorem 1.13]{deFig}. 
Suppose that $u$ changes sign so that  $u=u^+-u^-$ with both $u^{\pm}\not \equiv 0$,
$u^{\pm}\geq 0$. Then, 
\begin{align}\label{segnocostante}
\nonumber\lambda(m)=\frac{\int_\Omega H^p(\nabla u^+)+ \k \int_{\partial \Omega} 
(u^+)^p
+\int_\Omega H^p(-\nabla u^- )+ \k \int_{\partial \Omega} (u^-)^p}{\int_\Omega m (u^+)^p+\int_\Omega m (u^-)^p}
\\
\geq \min\left\{
\frac{\int_\Omega H^p(\nabla u^+)+ \k \int_{\partial \Omega} (u^+)^p
}{\int_\Omega m (u^+)^p}, \frac{\int_\Omega H^{p}(-\nabla u^- )+ \k \int_{\partial \Omega} (u^-)^p}{\int_\Omega m (u^-)^p}
\right\}.
\end{align}
Now, if the two quotients on the right hand side are equal, both $u^+$ and $-u^-$ are eigenfunctions and the strong maximum principle yields that  $u^+ >0$  and $u^->0$ a.e. on $\Omega$, which is impossible.
Otherwise,   the above quotients are different and in this case  
$u=u^{+}$ or $u=-u^{-}$. Finally, as $H$ is even,   $H(\nabla u^{-}) =H(-\nabla  u^{-})$, so that
we can always suppose that $u>0$, yielding  $\lambda^{+}(m)=\lambda(m)$.

On the other hand, since we are just assuming \eqref{h:posom} we cannot 
choose the sign of $\vfi$ a posteriori. This is the reason why we minimize in 
${\mathcal S}^{+}_{\k,m}$.
\end{remark}

In the rest of the paper we will denote by $\vfi$ the positive eigenfunction associated with the principal eigenvalue $\lambda^{+}(m)$ normalized to be in the unit sphere of $L^{p}(\Omega)$. 

\begin{proof}[Proof of Theorem \ref{thm:superlevel set}]
Let us first  observe that $\od^+$ defined in \eqref{Lambda} is achieved by a minimization 
argument taking into account \eqref{defM}, \eqref{h:convex}, and exploiting the Poincar\'e 
inequality in \cite[Lemma 3.1]{DG}.

In addition, it is possible to exploit the so called \emph{bathtub principle}, 
(see e.g.~\cite[Theorem~1.14]{lilo} or ~\cite[Lemma 3.3]{DG}), 
to obtain that the minimizing weight is bang-bang, namely
\[
m= \ind{D} -\beta\ind{D^c},
\] 
where $\ind{D}: \Omega \mapsto \{0,1\}$ is the characteristic function of the set $D\subset \Omega$ such that 
\begin{equation}\label{dinclusion}
\{ \varphi >t\} \subseteq D \subseteq \{ \varphi\geqslant t\}
\end{equation}
 for some $t>0$ and 
\begin{equation}\label{misura intervallo}
|D| = \frac{(\beta-m_0)|\Omega|}{1+\beta}.
\end{equation}
For what it concerns the last conclusion, we can apply Corollary 1.7 in \cite{ACF}
to get that $D=\{ \varphi >t\} $ up to a set of zero measure. 
\end{proof}

 \begin{remark}\label{rem:hc2}
 Corollary 1.7 in \cite{ACF} is proved only for dimension $N \ge 2$ and even $H$. Although  their proof can be adapted to the one dimensional case and without the hypothesis of symmetry of $H$, we will also give a different proof in dimension 1 by showing that $\vfi$ satisfies a monotonicity property
(see Theorem \ref{D:intervallo}).
 \end{remark}
 \begin{remark}\label{rem:min set}
As  a consequence of Theorem \ref{thm:superlevel set}, we deduce that
$\Lambda^{+}$ can be equivalently obtained as
\begin{equation}\label{eq:Lambda+nuova}
\Lambda^{+}:=\min \left\{\lambda^{+}(E), \, E\subset \Omega, \text{$E$ is measurable and }\; 0<|E|\leqslant \frac{(\beta-m_{0})|\Omega|}{1+\beta}\right\}
\end{equation} 
where, with a slight abuse of notation, we denote  $\lambda^{+}(E)=
\lambda^{+}(\ind{E}-\beta\ind{E^{c}}) $, and 
\begin{equation}
\Lambda^{+}=\lambda^{+}(D)=\lambda^{+}(\ind{D}-\beta\ind{D^{c}}), \quad
D=\{ \varphi >t\}
\end{equation}
where $D$ satisfies \eqref{misura intervallo} and $\varphi$ is the eigenfunction associated with $\lambda^{+}(D)$. 
We will refer to any set $D$ which solves \eqref{eq:Lambda+nuova} as the \textit{optimal set}.  
 \end{remark}

In analogy to what is known in the context of fully non-linear operators, such as Pucci operators where principal half-eigenvalues are studied (see e.g. \cite{BD,QS}),  we can define another principal eigenvalue,  
$\lambda^-(m)$, as follows

\begin{definition}\label{def:lambda-}
We set 
\begin{equation}
\label{minpbmeno} 
\lambda^-(m)=\inf_{u\in {\mathcal S_{\k, m}^-}}{\mathcal R_{\k,m}}(u)
\end{equation}
where
\begin{align*}
\label{skmeno}{\mathcal S_{\k, m}^-} 
:=\left\{u\in W^{1, p}(\Omega),\ u \leqslant 0, \, \int_\Omega m(x)|u|^p>0\right\}, 
\end{align*}
for $\k<+\infty $ and 
\[
\mathcal S_{\infty, m}^-:=\left\{  u\in W^{1, p}_{0}(\Omega),\ u \leqslant 0, \, \int_\Omega m(x)|u|^p>0\right\}.
\]
in the case $\k=+\infty$.
\end{definition} 
The counterpart of Proposition \ref{prop:lambda} and Theorem \ref{thm:superlevel set} is contained in the following result. 
\begin{proposition}\label{prop:exlameno}
For any fixed $m \in \mathcal{M}$,  $\lambda^-(m)$ is achieved by a negative eigenfunction.
Moreover, all the other conclusions of  Proposition \ref{prop:lambda} hold with obvious changes. 

In addition,  the  minimization problem
\[ \Lambda^-:=\inf_{m \in \mathcal{M}} \lambda^-(m) \]
has a bang-bang solution $m_{-}=\ind{D_{-}}-\beta\ind{D_{-}^{c}}$. If $\vfi_{-}$ is an associated negative eigenfunction there exists $t>0$ with $D_{-}=\{\vfi_{-}<-t\}$.
\end{proposition}
\begin{proof}
As a preliminary observation, notice that $u <0$ is a minimizer for \eqref{minpbmeno} if and only if
\[ \widetilde u(x)=-u(x)>0 \]
solves  
\[
\inf \left \{ \int_{\Omega} (\widetilde H(\nabla v))^p + \k \int_{\partial \Omega} |v|^p : v \in W^{1, p}(\Omega), \, v \geq 0,\, \int_{\Omega} m(x) \abs{v}^p =1 \right \}
\]
where $\widetilde{H}(\xi)= H(-\xi)$. Since $\widetilde{H}$ satisfies \eqref{h:norma}, 
\eqref{h:posom}, \eqref{h:convex}, we can apply Proposition  \ref{prop:lambda} and Theorem \ref{thm:superlevel set} to $\widetilde{H}$, obtaining the conclusion.
\end{proof}

If $H$ is even, then  $\lambda^+(m)=\lambda^-(m)$, and the corresponding eigenfunctions  $\vfi_+>0$, $\vfi_-<0$ satisfy $\vfi_+=-\vfi_-$, so that it is always possible to choose a positive eigenfunction generating the whole eigenspace.
In particular, this is true for instance in the case of the $p$-Laplace operator, and $\lambda^+(m)=\lambda^-(m)$ coincides with the usual notion of principal eigenvalue.  Under our assumptions, $H$ is not even in general, for instance,  in dimension one, $\widetilde H=H$ if and only if $a=b$, see \eqref{h:1dim}.

\section{Proofs of the main results}\label{localization}

In this section we will provide the proofs of Theorems \ref{thm:localization} and \ref{thm:Lambda+=Lambda-}. 
First we study the position of the optimal set in the one-dimensional case.
Let $H$ be given  in \eqref{h:1dim} and  consider the following problem
\begin{equation}\label{prob:N=1}
\begin{cases}
-\left((H( u'))^{p-1} H'( u')\right)' = \lambda m(x) |u|^{p-2}u & \text{ in } (0,1)\\
 H^{p-1}( u'(1)) H'( u'(1))+\k|u(1)|^{p-2}u(1)=0&
 \\
  -H^{p-1}( u'(0)) H'( u'(0))+\k|u(0)|^{p-2}u(0)=0 & 
\end{cases}
\end{equation}

In what follows we denote by $m=\ind{D}-\beta\ind{D^{c}}$ the minimizer for \eqref{Lambda}, and  $\varphi$ the positive eigenfunction corresponding to $\lambda^+(m)$, normalized with respect to the $L^{p}$-norm. Recall that $D$ is a minimizer of \eqref{eq:Lambda+nuova}, so that we will refer to it as \textit{optimal}. 
\begin{theorem}\label{D:intervallo}
The following conclusions hold.
\begin{enumerate}
\item 
 If $\k \in (0, +\infty]$, then $\vfi$ attains its maximum in $\alpha \in (0, 1)$, and $\vfi$ is strictly increasing in $(0, \alpha)$, and strictly decreasing in $(\alpha, 1)$. 
\item 
If $\k=0$,  
then $\vfi$ is  monotone. 
\item 
The optimal set $D=\{x\in (0,1) \,:\; \vfi(x)>t\}$ is an interval.
\item 
The set $\{x\in (0,1) :  \vfi'(x)=0 \}$ is finite.
\end{enumerate}
\end{theorem}
\begin{proof}
We will follow the argument of  \cite[Proposition 4]{LLNP} and we start proving
conclusion \textit{$(1)$.} 
Notice that, by elliptic regularity, $\vfi'\in C([0,1])$, so that,
if $\k>0$ the boundary conditions and \eqref{h:1dim} immediately imply that $\vfi$ achieves its maximum in $(0,1)$.
Denoting with $\alpha$ the first maximum point of $\vfi$, one can use the monotone  rearrangements (see Section \ref{rearrangements}) and define
\[
\vfi^{R}=\begin{cases}
\vfi^{\ast} &x\in (0,\alpha)
\\
\vfi _{\ast} &x\in (\alpha,1)
\end{cases}\,,
\qquad m^{R}=\begin{cases}
m^{\ast}&x\in (0,\alpha)
\\
m_{\ast} &x\in (\alpha,1)
\end{cases}
\]
where, in view of Remark \ref{rem:I},  
$\vfi^{\ast}$, ($\vfi_{\ast}$)  stands for  the monotone increasing  (decreasing) rearrangement of  the restriction of the function $\vfi$  to the interval $(0,\alpha)$
($(\alpha,1)$) and analogously for $m$. As $\alpha$ is  a maximum point,  
$\vfi^{R} \in H^{1}(0,1)$.

The Hardy--Littlewood inequality (see \cite{Kawohl}) implies 
\begin{align*}   
\int_0^\alpha m \varphi^p &= \int_0^\alpha (m+\beta) \varphi^p - \beta \int_0^\alpha \varphi^p \leqslant \int_0^\alpha (m+\beta)^* (\varphi^*)^p - \beta \int_0^\alpha (\varphi^*)^p 
=
 \int_0^\alpha m^* (\varphi^*)^p 
\end{align*}
and an analogous inequality holds for $\varphi_{*}$ and $m_{*}$. 
Moreover,
\[ (\varphi^*)^p(0)=\min_{[0, \alpha]} \varphi^p \leqslant  \varphi^p(0), \quad (\varphi_*)^p(1)=\min_{[ \alpha, 1]} \varphi^p \leqslant  \varphi^p(1). 
\]
Note that $m^{R}$ and $\varphi^{R}$ are admissible competitors for $\Lambda^{+}$, and Proposition \ref{polya monotona} yields
\begin{align*}
\Lambda^+ \leqslant  \mathcal {R}_{\k,m^{R}}(\varphi^{R})
&= \dfrac{\int_0^\alpha H^p((\varphi^*)') +\int_\alpha ^1 H^p((\varphi_*)')+ \k (\varphi^*)^p(0) + \k (\varphi_*)^p(1)}{ \int_0^\alpha m^{*} (x) (\varphi^*)^p+ \int_\alpha^1 m_{*} (x) (\varphi_*)^p}  
 \\
 &\leqslant \frac{\int_0^1 H(\varphi')^p +\k \varphi^p(0) + \k \varphi^p(1)}{ \int_0^1 m(x) \varphi^p} =\Lambda^+,
\end{align*}
which implies
\[ 
\frac{\int_0^\alpha H^p((\varphi^*)') +\int_\alpha ^1 H((\varphi_*)')^p+ \k (\varphi^*)^p(0) + \k (\varphi_*)^p(1)}{ \int_0^\alpha m^{*}(x) (\varphi^*)^p+ \int_\alpha^1 m_{*}(x) (\varphi_*)^p} =\frac{\int_0^1 H^p(\varphi') +\k \varphi^p(0) + \k \varphi^p(1)}{ \int_0^1 m(x) \varphi^p}.
\]
As a consequence,
\[
\int_{0}^{\alpha}H^{p}(\vfi')=\int_{0}^{\alpha}H^{p}((\vfi^{*})'),
\qquad\text{and}\qquad
\int_{\alpha}^{1}H^{p}(\vfi')=\int_{\alpha}^{1}H^{p}((\vfi_{*})').
\]
Then, Proposition \ref{equality polya} implies that $\vfi=\vfi^{R}$ yielding  that $\vfi$ increases 
up to its maximum and then decreases.

Let us now show conclusion \textit{$(2)$.} 
Consider the decreasing rearrangements $\varphi_*$ and $m_*$. 
Then, arguing as before,
\[ 
\Lambda^+= \frac{\int_0^1 H(\varphi')^p }{ \int_0^1 m(x) \varphi^p} \ge\frac{\int_0^1 H((\varphi_*)')^p }{ \int_0^1 m_*(x) \varphi_*^p}
 \geqslant \Lambda^+. 
\]
Therefore, equality holds, and applying Proposition \ref{equality polya} one obtains that $\varphi$ is monotone. 

Conclusion \textit{$(3)$} directly follows from the previous ones.

Let us now prove conclusion $(4)$.
Notice that, for any $x<y \in D$, integrating the equation in \eqref{prob:N=1} in $(x,y)$, one has 
\[ 
(H^{p})'(\varphi'(x)) - (H^{p})'(\varphi'(y)) = \Lambda^+  \int_x^y \varphi^{p-1} >0. 
\]
Taking into account  that the function $H^{p}$ is strictly convex, we have that $\varphi'(x) > \varphi'(y)$ for any $x<y \in D$, namely $\varphi'$ is strictly decreasing in $D$. Similarly, one proves  that $\varphi'$ is strictly increasing on every connected component of  $D^c$. This shows that $\vfi$ has a finite number of critical points since $D$ is an interval. 
As an immediate consequence, we also get that the monotonicity of $\vfi$ is strict in the intervals $(0,\alpha)$, $(\alpha,1)$.
\end{proof}

We are now ready to prove Theorem \ref{thm:localization}. 
\begin{proof}[Proof of Theorem \ref{thm:localization}]
The fact that $D$ is an interval has already been proved in Theorem \ref{D:intervallo}.

 Let us  deal with the Neumann case first, namely $\k=0$. 
 Due to Theorem \ref{thm:superlevel set} and  Theorem \ref{D:intervallo}, we know that the optimal eigenfunction $\vfi$ is strictly increasing or decreasing and recalling that
the optimal interval $D$ is the positivity set of the optimal weight $m$,  one has
the following alternative
\[ 
m_{(0, c)}:= \ind{(0, c)} - \beta \ind{(c, 1)} \quad \text{ or } \quad m_{(1-c, 1)} := \ind{(1-c, 1)} - \beta \ind{(0, 1- c)}, \]
where $c:=|D|$.
We denote $\lambda^{+}(E)=\lambda^{+}(\ind{E}-\beta\ind{E^{c}}) $ (see Remark \ref{rem:min set}). 

Let us first deal with the case $a>b$ and suppose by contradiction that $m_{(1-c, 1)} := \ind{(1-c, 1)} - \beta \ind{(0, 1- c)}$ is the optimal weight.  Observe that Theorem \ref{D:intervallo} yields that the optimal eigenfunction $\varphi$ is monotone and the contradiction hypothesis readily implies that $\varphi$ has to be increasing. Defining $\psi(x)=\varphi(1-x)$ and taking into account \eqref{h:1dim}, one has
\[ \int_0^1H^p(\psi') = \int_0^1 H^p(-\varphi'(1-x))=
 b^p \int_0^1 (\varphi'(1-x))^p =b^p \int_0^1 (\varphi')^p=\frac{b^p}{a^p} \int_0^1 H^p(\varphi'). \]
Also
\[ \int_0^1 m_{(0, c)} \psi^p= \int_0^1 m_{(1-c, 1)} \varphi^p. \]
Hence
\[ \lambda^+((0,c)) \leqslant {\mathcal R_{0,m_{(0,c)}}}(\psi)
=\frac{b^p}{a^p} \lambda^+((1-c, 1))= \frac{b^p}{a^p} \Lambda^{+} <\Lambda^{+}
\]
that contradicts the minimality of $\Lambda^{+}$.
The case $a<b$ follows analogously.

We now take into account the Dirichlet case $\k=\infty$. 
Applying Proposition \ref{prop: polya anis} and \cite[Proposition 2.28]{VanSchaft} one deduces that
\[
\Lambda^+ \leqslant \frac{\int_{I} \tau_{l}\left( H((\varphi^\star)')^p\right)}{\int_{I} \tau_l(m^\star |\varphi^\star|^p)} \leqslant \frac{\int_{I^\star} H((\varphi^\star)')^p}{\int_{I^\star} m^\star |\varphi^\star|^p} \leqslant \frac{\int_I H(\varphi')^p}{\int_I m |\varphi|^p}= \Lambda^+, 
\]
where $\tau_l$ denotes the translation operator in the direction $l=-\frac{a}{a+b}$ and $\varphi^\star$ is the anisotropic rearrangement of $\vfi$ with respect to $H_0$, see Section \ref{sec:anissym}. Exploiting the equality case in Proposition \ref{prop:  polya anis},  we get that $\vfi(x+\frac{a}{a+b})=\vfi^\star(x)$.
Then,
\[
\left\{\vfi\left(x+\frac{a}{a+b}\right)>t\right\}=\left(\frac{-a|D|}{a+b}, \frac{b|D|}{a+b}\right),
\]
completing the proof.
\end{proof} 

\begin{remark}
Notice that the presence of the anisotropy forces the position of the optimal interval $D$. Indeed, in the Neumann case $D$ lies on the left or on the right of $(0,1)$ if $a>b$ or $b>a$ and in the Dirichlet case  the centre of the optimal interval given  in Theorem \ref{thm:localization} is
\[ 
x_{0}=\frac{|D|(b-a)+2a}{2(a+b)}.
 \]
 and $x_{0}<1/2$ if $b >a$, while $x_{0}>1/2$ if $b<a$. 
 In both cases, if $a=b$ we recover the known results of \cite{CC89,louyan,DG}.
 \end{remark}

Now we prove Theorem \ref{thm:Lambda+=Lambda-}.

\begin{proof}[Proof of Theorem \ref{thm:Lambda+=Lambda-}]
Without loss of generality we may assume that the origin is the center of symmetry for $\Omega$. Let $\vfi_+$ the positive eigenfunction corresponding to $\lambda^+(m_+)=\Lambda^+$, hence
\[ \lambda^+(m_+)=\frac{\int_\Omega H^p(\nabla \vfi_+) + \k \int_{\partial \Omega} \vfi_+^p}{\int_\Omega m_+ \vfi_+^p}. \]
We now define
\[ v(x)=-\vfi_+(-x) <0, \]
thus $\nabla v(x)=(\nabla \vfi_+)(-x)$. 
Therefore,
\begin{align*} 
\Lambda^+=\lambda^+(m_+)&
=\frac{\int_{\Omega} H^p(\nabla v)(-x) + \k \int_{\partial \Omega} |v|^p(-x)}{\int_\Omega m_+(x)|v|^p(-x)} 
=\frac{\int_{\Omega} H^p(\nabla v)(x) + \k \int_{\partial \Omega} |v|^p(x)}{\int_\Omega m_+(-x)|v|^p(x)} \\
&\ge \lambda^-(m_+(-x)) \ge \Lambda^-,
\end{align*}
where we have used that $m_+(-x) \in \mathcal{M}$ and $v\in \mathcal{S}_{\k, m}^-$. 

Analogously, take $\vfi_-$ the eigenfunction such that $\lambda^-(m_-)=\Lambda^-$. Then
\[ \Lambda^-=\lambda^-(m_-) \ge \lambda^+(m_-(-x)) \ge \Lambda^+, \]
from which we immediately deduce $\Lambda^+=\Lambda^-$.
\end{proof}
\begin{remark}
The argument of Theorem \ref{thm:Lambda+=Lambda-} actually shows that
 if $\Omega=-\Omega$ and $m(x)=m(-x)$, then 
$\lambda^{+}(m)=\lambda^{-}(m)$ with associated eigenfunction $\vfi_{+}(x)=-\vfi_{-}(-x)$.
Indeed,
 let $m$ satisfy $m(x)=m(-x)$. If we set $v(x):=-\vfi_+(-x)$ we get 
\[ \lambda^+(m)= \frac{\int_\Omega H^p(\nabla \vfi_+) + \k\int_{\partial \Omega} \vfi_+^p}{\int_\Omega m(x) \vfi_+^p } = \frac{\int_\Omega H^p(\nabla v) + \k\int_{\partial \Omega} |v|^p}{\int_\Omega m(x) |v|^p } \ge \lambda^-(m). \]
Similarly, we prove the opposite inequality. Thus $\lambda^+(m)=\lambda^-(m)$, and $-\vfi_+(-x)$ is a minimizer for $\lambda^-(m)$, from which 
we deduce $\vfi_+(x)=-\vfi_-(-x)$  by uniqueness of the positive eigenfunction (see Proposition \ref{prop:lambda}).
\end{remark}

From the proof above we  also deduce that if $u, m$ is a minimizer for $\Lambda^+=\Lambda^-$, then $-u(-x), m(-x)$ is still a minimizer. 

We now consider the case $N=1$, and $\Omega=[0, 1]$. Notice that the reflection above in this case reads as $x \mapsto 1-x$. Therefore, what we proved until now can be stated as follows: if $u, m$ is a minimizer for $\Lambda^+=\Lambda^-$, then $-u(1-x), m(1-x)$ is still a minimizer. 
Due to Theorems \ref{thm:localization}, we can actually say something more, namely that these are the only minimizers for $\Lambda^+=\Lambda^-$, at least if we take Dirichlet or Neumann boundary conditions. 
\begin{proposition}\label{prop:relaz m+m-}
Let $N=1$, and consider $\k=+\infty$ (Dirichlet), or $\k=0$ (Neumann) with $a \ne b$. Let us denote $m_+$ a weight such that $\Lambda^+=\lambda^+(m_+)$ and $\vfi_+>0$ the associated eigenfunction, and similarly  $m_-$ a weight such that $\Lambda^-=\lambda^-(m_-)$ with eigenfunction $\vfi_-<0$. Then 
\[ m_+(x)=m_-(1-x) \]
and
\[\vfi_+(x)=-\vfi_-(1-x). \]
\end{proposition}

We preliminary notice that, calling $\widetilde H(x):=H(-x)$, then the polar function of $\widetilde H$ (see for instance \cite{AlvinoFLT}) is 
\begin{equation}\label{eq:H-zero-tilde}
 \widetilde H_0(x):=\sup_{x \in \R} \frac{ \langle t, x \rangle}{\widetilde H(x)} =H_0(-x)= \begin{cases}
\frac x b & \text{ if } x \ge 0\\
- \frac x a & \text{ if } x < 0. 
\end{cases}
\end{equation}

\begin{proof} 
\textit{Dirichlet case. } Let $\vfi_+$ the positive eigenfunction corresponding to $\lambda^+(m_+)=\Lambda^+$ so that the associated optimal set  $D_+$ satisfies \eqref{localization D}.

Let us now consider  $\vfi_-$ the negative eigenfunction corresponding to $\lambda^-(m_-)=\Lambda^-$ (see Proposition \ref{prop:exlameno}).  As a consequence of Corollary \ref{lem:anis polya 2}, reasoning as in Theorem  \ref{thm:localization}, we get 
$m_-=\ind{D_-} - \beta \ind{D_-^c}$, where $D_-$ is defined by 
\[ D_-= \left( \frac{(1-c)b}{a+b}, \frac{ca +b}{a+b} \right), \quad  \text{ with  $c:=|D_-|$ given by \eqref{misura intervallo},} \]
which, in view of \eqref{localization D}, gives 
\[ m_+(x)=m_-(1-x). \]
Now, observe that
\[ -\Delta_{H,p} \vfi_+=\Lambda^+ m_+ \vfi_+^{p-1}, \]
and
\[ - \Delta_{H, p} v=\Lambda^+ m_+ v^{p-1}, \quad \text{ where $v(x):=-\vfi_-(1-x)$. } \]
Thus, using Proposition \ref{prop:lambda} conclusion (1), we get
\[ \vfi_+(x)=-\vfi_-(1-x). \]

\textit{Neumann case.}  Let us consider the case $a > b$.  We know by Theorem \ref{thm:localization} that $\vfi_+$ is decreasing and $m_+=\ind{D_+} - \beta \ind{D_+^c}$, where $D_+$ is $(0, |D_+|)$. Following exactly the same argument, one shows that also $\vfi_-$ is decreasing, and $m_-=\ind{D_-} - \beta \ind{D_-^c}$, where $D_-=(1-|D_-|, 1)$. This immediately implies $m_+(x)=m_-(1-x)$. As above, we get $\vfi_+(x)=-\vfi_-(1-x)$. If $a < b$ one can argue analogously.
\end{proof}

\section{Anisotropic rearrangement inequalities in $\R$.}\label{rearrangements}

In this section we prove all the rearrangement inequalities that have been exploited  to
obtain the qualitative properties of the optimal set $D$ in the one dimensional case.
The next subsection deals with the monotone rearrangements while Subsection
\ref{sec:anissym} treats anisotropic symmetrizations.

\subsection{Monotone rearrangements}\label{monotonere}
Given a function $u:[0, 1] \to \R^+$,  
we define 
the monotone decreasing rearrangement of $u$ as the function  $u_*: [0, 1] \to \R^+$  such that 
\[
u_*(x)=
\begin{cases}
 \sup  u  & \text{ if } x=0\\
\inf \left\{ t : \, \mu_u(t) < x \right\} &\text{ if } x \in (0, 1],
\end{cases}
\]
where 
\begin{equation}\label{distribution function} \mu_u(t):=\left| \left \{ x \in [0, 1]: \, u(x)  >t \right\} \right| \end{equation}
is the distribution function of $u$. 
The monotone increasing rearrangement $u^*$ is defined analogously.

Our first aim is to show a Polya type inequality as stated in  the following result.
\begin{proposition}\label{polya monotona}
Let $H$ be defined as in \eqref{h:1dim} and $u \in W^{1,p}(0, 1)$.
Then
\begin{equation}\label{polya:ineq}
 \int_0^1 H^p(u') \ge \int_0^1 H^p((u_*)'). 
 \end{equation}
\end{proposition}
\begin{proof}
The proof  closely follows the arguments of  \cite[Section II.3]{Kawohl}   (see in particular Lemma 2.4, Lemma 2.6) , so that we will simply enlighten the differences in our situation.
In view of \cite[ Remark 2.20]{Kawohl}, we can assume that $u > 0$, piecewise affine and with maximum at the origin.
 
Let us consider the set of the values of $u$ at the non-differentiability points, together with $u(0)$ and $u(1)$, denote them by $\{a_1 \leqslant \dots \leqslant a_k\}$ and let 
 \[ 
D_i= \left\{ x \in [0, 1]: a_i < u(x) < a_{i+1} \right\} \;, \;\; E_i= \left\{ x \in[0, 1]: a_i < u_*(x) < a_{i+1} \right\}.
\]

Note that $E_{i}$ are always connected, while $D_{i}$ may be not. Arguing as in \cite[Lemma 2.4]{Kawohl},  for every $i=1,\dots, k$, it results that
\begin{equation}\label{eq:decomposition-Di}
D_{i}=\bigcup_{j=1}^{N(i)}Y_{ij},
\end{equation}
where $u$ is differentiable and non-constant in each $Y_{ij}$. We assume, without loss of generality, that for each $i$, $Y_{ij}$ are ordered according to their distance from the origin.  Moreover, since $H(0)=0$, it is enough to prove that
\begin{equation}\label{eq:goal1}
\int_{D_i} H^p(u') \ge \int_{E_i} H^p\left((u_*)'\right), \qquad \text{for $i=1,\dots, k.$}
\end{equation}

Since $u$ is injective in every $Y_{ij}$, for each $\lambda \in (a_i, a_{i+1})$ the equation $u(x)=\lambda$ has a unique solution, so that we can define the differentiable function $\rho_{j}:(a_{i},a_{i+1})\mapsto Y_{ij}$
such that
\begin{equation}
\label{eq:infoRhoj}
u(x)=\lambda, \quad\text{ if and only if }\quad x=\rho_{j}(\lambda)\;,\;\; u'(\rho_j(\lambda)) = \left( \rho'_j (\lambda) \right)^{-1}. 
\end{equation}

Hence
\begin{equation}\label{eq:deco1}
\int_{D_i} H^p(u'(x))dx= 
\sum_{j=1}^{N(i)} \int_{a_i}^{a_{i+1}} 
H^p\left[
\left(  \rho'_j (\lambda) \right)^{-1} \right] \abs{\rho'_j (\lambda)} \, d\lambda,
\end{equation}
where we have taken into account that the sign of $\rho'_{j}$ allows to obtain the integral from  $a_{i}$ up to  $a_{i+1}$.
At the same time, as $u_*$ is monotone we can define
$\rho_{*}: (a_{i},a_{i+1})\mapsto E_{i}$ such that
\begin{equation}\label{eq:ustar}
u_*(x)=\lambda \quad\text{ if and only if }\quad x=\rho_{*}(\lambda)
\quad \text{and}\quad (u_*)'(\rho_*(\lambda)) = \left[(\rho_*)'(\lambda) \right]^{-1}. 
\end{equation}
Thus, also using the hypothesis that $u$ has its maximum in the origin, 
it results
 \[
 \sign \rho'_{j}(\lambda)=  \sign  u'(x) = (-1)^j, \text{ in } Y_{ij}, \quad \text{for $j=1,\dots, N(i)$ and  $i=1,\dots, k$,}
 \]
so that
\begin{equation}\label{eq:3}
 \abs{\sum_{j=1}^{N(i)} (-1)^{j+1} \rho'_j(\lambda)} 
 = \abs{\sum_{j=1}^{N(i)} \left(-\left| \rho'_j(\lambda) \right| \right)} 
 = \sum_{j=1}^{N(i)}\abs{\rho'_j (\lambda)}. 
 \end{equation}
Furthermore, using the definition of $u_{*},$ it results
\begin{equation}\label{eq:rhostar} 
\rho_*(\lambda)=
\begin{cases}
\dys \sum_{j=1}^{N(i)} (-1)^{j+1} \rho_j(\lambda) & \text{ if $N(i)$ is odd}, \\
\dys \sum_{j=1}^{N(i)} (-1)^{j+1} \rho_j(\lambda)+1 & \text{ if $N(i)$ is even}.
\end{cases} 
\end{equation}
As a consequence, performing the change of variable $x=\rho_{*}(\lambda)$ and using \eqref{eq:ustar}
\begin{equation*}\label{eq:deco2}
\begin{split}
\int_{E_i} H^p((u_*(x))')dx &= 
\int_{a_i}^{a_{i+1}} H^p\left[ 
\left( \sum_{j=1}^{N(i)} (-1)^{j+1} \rho'_j  (\lambda) \right)^{-1} \right] 
\abs{\sum_{j=1}^{N(i)} (-1)^{j+1} \rho'_j  (\lambda)} \, d\lambda 
\\
&=
\int_{a_i}^{a_{i+1}} H^p\left[ 
\left( \sum_{j=1}^{N(i)} (-1)^{j+1} \rho'_j  (\lambda) \right)^{-1} \right] 
 \sum_{j=1}^{N(i)}\abs{\rho'_j (\lambda)} d\lambda.
\end{split}
\end{equation*}
Then, recalling \eqref{eq:deco1}, in order to obtain \eqref{eq:goal1},   it is 
 enough to show that  
\[ \sum_{j=1}^{N(i)} \int_{a_i}^{a_{i+1}} 
H^p\left[\left(  \rho'_j (\lambda) \right)^{-1} \right] \abs{ \rho'_j (\lambda)} d\lambda
 \ge 
 \int_{a_i}^{a_{i+1}}  H^p\left[ \left( \sum_{j=1}^{N(i)}
  (-1)^{j+1}  \rho'_j  (\lambda) \right]^{-1} \right) \sum_{k=1}^{N(i)}  
  \abs{\rho'_k (\lambda)}d\lambda. 
\]
In turn, it is sufficient to show that the following point-wise inequality holds
\begin{equation}\label{disug} 
\sum_{j=1}^{N(i)}  \alpha_j H^p\left[
\left( \rho'_j (\lambda) \right)^{-1} \right]  \ge 
H^p \left[ \left( \sum_{j=1}^{N(i)} (-1)^{j+1} \rho'_j  (\lambda) \right)^{-1} \right], 
\end{equation}
where 
\[ 
\alpha_j= \abs{  \rho'_j (\lambda)}  \left( \sum_{k=1}^{N(i)} \abs{ \rho'_k (\lambda)} \right)^{-1}. 
\]
In order to prove \eqref{disug} first notice that, as $u(0)=\max_{[0, 1]}u$, 
the expression \eqref{h:1dim} yields
\begin{align*} 
\sum_{j=1}^{N(i)} \alpha_j H^p \left[
\left( \rho'_j (\lambda) \right)^{-1} \right] 
&= 
a^p \sum_{j \text{ even}} \alpha_j  \abs{  \rho'_j (\lambda) }^{-p} + b^p 
\sum_{j \text{ odd}} \alpha_j \abs{\rho'_j (\lambda) }^{-p}
\\
&
= \left( \sum_{k=1}^{N(i)} \abs{\rho'_k(\lambda)} \right)^{-1} 
\left[ a^p \sum_{j \text{ even}}  \abs{\rho'_j (\lambda)}^{1-p} + b^p 
\sum_{j \text{ odd}}  \abs{  \rho'_j (\lambda)}^{1-p}  \right] .
 \end{align*}
On the other hand, in view of \eqref{eq:3} we have
\[ 
H^p\left[ \left( \sum_{j=1}^{N(i)} (-1)^{j+1} \rho'_j  (\lambda) \right)^{-1} \right] 
= 
H^p \left[- \left( \sum_{j=1}^{N(i)} \abs{\rho'_j (\lambda)} \right)^{-1} \right] 
= b^p  \left(\sum_{j=1}^{N(i)} \abs{ \rho'_j (\lambda)} \right)^{-p} . 
\]
Then, \eqref{disug} holds if
\begin{equation}\label{disug2}
a^p \sum_{j \text{ even}}  \abs{\rho'_j (\lambda)}^{1-p} + b^p \sum_{j \text{ odd}}  \abs{  \rho'_j (\lambda)}^{1-p} \ge b^p  \left(  \sum_{j=1}^{N(i)} \abs{ \rho'_j (\lambda)} \right)^{1-p}. 
\end{equation} 
Notice that since $p >1$, for any $j=1,\dots, N(i) $ and $i=1,\dots,k$ it holds

\begin{equation*}\label{ineq} 
\abs{  \rho'_j (\lambda)}^{1-p} \ge \left(  \sum_{j=1}^{N(i)} \abs{\rho'_j(\lambda)} \right)^{1-p}.
\end{equation*}
Thus, \eqref{disug2} is proved if 
\[
a^p \sum_{j \text{ even}}  \abs{\rho'_j (\lambda)}^{1-p} + b^p \sum_{j \text{ odd}}  \abs{  \rho'_j (\lambda)}^{1-p} \ge b^p   \abs{  \rho'_1 (\lambda)}^{1-p}, 
\]
which is evidently true.
\end{proof}
In the following proposition we analyze the equality case in \eqref{polya:ineq}.
\begin{proposition}\label{equality polya} 
Assume \eqref{h:1dim}. If $u\in  W^{1,p}(0,1)$ is such that
\begin{equation}\label{polya:eq} 
\int_0^1 H^p(u') = \int_0^1 H^p((u_*)'),
\end{equation}
then $u$ is monotone. 
\end{proposition}

Let us notice that if $u$ is piece-wise affine, the result follows from an inspection of the proof above. 
Indeed, if we argue by contradiction and assume that $u$ is not monotone, then there exists an index $i$ such that $N(i) \ge 2$. Thus, the inequality \eqref{disug2} holds with a strict sign as $p >1$. Going backwards up to the beginning of the proof of Proposition \ref{polya monotona} one realizes that
 the inequality \eqref{polya:ineq} would have a strict sign too, which contradicts the 
 hypothesis. 
 As a consequence, the equality forces $N(i)=1$ for any $i=1,\dots,k$, namely that $u$ is decreasing. 
We now consider the case $u \in W^{1,p}(0, 1)$, 
 following an analogous argument used in the isotropic case (see for example \cite{BLR, louyan}) . 
\begin{proof}[Proof of Proposition \ref{equality polya}]
First notice that, as $u\in W^{1,p}(0, 1)$, $u$ is continuous.
Let us argue by contradiction and assume that there  exist $t_1\neq t_2\in (0, 1)$ such that $u(t_1)=u(t_2)$ and $t_{1},$ $ t_{2}$ are not local maximum nor minimum points. Then 
\begin{equation}\label{eq:max}
\min_{(0,k)} u < \max_{(k,1)} u, \quad \text{and}\quad
\max_{(0, k)} u > \min_{(k, 1)} u\quad \text{for all $k \in (t_1, t_2)$.}
\end{equation}
 Let us define  $u_{1}:(0,k)\mapsto \R$ by  $u_1(t)=u(t)$  and 
 $u_{2}:(0, 1-k)\mapsto \R$ by $u_2(t)=u(k+t)$.
Then we can consider
\[ v(t)= \begin{cases}
(u_1)_*(t) & t \in (0, k) \\
(u_2)_*(t-k) & t \in (k, 1). 
\end{cases} 
\]
Observe that $v_*=u_*$. Hence, hypothesis \eqref{polya:eq}  and Proposition \ref{polya monotona} yield 
\begin{align*} 
\int_0^1 H^p(u') &= \int_0^1 H^p((u_*)') =
\int_0^1 H^p((v_*)') 
\leqslant \int_0^1 H^p(v')   
\\
&= 
\int_0^k H^p(((u_1)_*)') + \int_0^{1-k} H^p(((u_2)_*)')  \leqslant 
\int_0^k H^p(u_1') + \int_0^{1-k} H^p(u_2')  
\\
&
= \int_0^k H^p(u'(t))dt + \int_0^{1-k} H^p(u'(k+t))dt  =\int_0^1 H^p(u').  
\end{align*}
In particular,
\begin{equation}\label{eq:A}
\begin{split}
\int_0^1 H^p((u_*)') & -\int_0^k H^p(((u_1)_*)')- \int_0^{1-k} H^p(((u_2)_*)') \\
&  = \int_0^1 H^p((v_*)') - \int_0^1 H^p(v')=0. 
\end{split}
\end{equation}
Define $\sigma=\min v_*$, $\Sigma=\max v_*$, and similarly $\sigma_i, \Sigma_i$ for $(u_i)_*$.
Then \eqref{eq:max} yields 
\[ \sigma_1 < \Sigma_2, \quad \sigma_2<\Sigma_1. \]
Let us denote by $\mu $ the distribution function of $v$, see \eqref{distribution function}, so that $(v_*)'(t)=1/\mu'(v_*(t))$ and $\mathcal{E}\subset (0,1)$ such that $(v_*)'\equiv 0$ on $\mu(  \mathcal{E})$. Analogously define $\mu_{i}$ and $\mathcal{E}_{i}$ for $u_{i}$.  Notice that $\mathcal{E}=\mathcal{E}_1 \cup \mathcal{E}_2$. As $v_{*}$ is not increasing, performing a change of variable, we obtain 
\begin{equation}\label{change var}
\int_0^1 H^p((v_*)')= - \int_{(\sigma, \Sigma) \setminus \mathcal{E}} H^p\left(\frac{1}{\mu'(s)}\right) \mu'(s) \, ds. 
\end{equation}
Thus \eqref{eq:A} becomes
\begin{align*}
\int_{(\sigma_1, \Sigma_1) \setminus \mathcal{E}_1} H^p\left(\frac{1}{\mu_1'(s)}\right) \mu_1'(s)  + \int_{(\sigma_2, \Sigma_2) \setminus \mathcal{E}_2} H^p\left(\frac{1}{\mu_2'(s)}\right) \mu_2'(s) - \int_{(\sigma, \Sigma) \setminus \mathcal{E}} H^p\left(\frac{1}{\mu'(s)}\right) \mu'(s) =0.
 \end{align*}
Notice that $\sigma=\min\{ \sigma_1, \sigma_2 \}$ and $\Sigma=\max\{\Sigma_1, \Sigma_2 \}$. Also, setting $\overline \sigma= \max\{\sigma_1, \sigma_2\}$ and $\underline \Sigma=\min \{\Sigma_1, \Sigma_2 \}$, we have $\sigma \leqslant \overline \sigma < \underline \Sigma \leqslant \Sigma$.
 On $(\sigma, \overline \sigma) \cup (\underline \Sigma, \Sigma)$, 
 one of the $\mu_{i}$ is constant, if for example this occurs for $\mu_{1}$, then 
 $\mu'=\mu_{2}'$;
and in the interval $(\overline \sigma, \underline \Sigma)$ it holds $\mu'=\mu_1'+\mu_2'$. Thus the above equality becomes
\begin{equation}\label{eq:A=0}
\int_{(\overline \sigma, \underline \Sigma)\setminus (\mathcal{E}_1 \cup \mathcal{E}_2)} \left[ H^p\left(\frac{1}{\mu_1'}\right )\mu_1'+H^p\left(\frac{1}{\mu_2'}\right )\mu_2' -H^p\left(\frac{1}{\mu_1'+\mu_2'}\right )(\mu_1'+\mu_2')  \right] =0. 
\end{equation}
We now use the positive homogeneity of $H$ and the fact that $\mu_i' < 0$ to observe that the integrand is equal to 
\[ H^p(-1) \left( -\abs{\mu_1'}^{1-p} -\abs{\mu_2'}^{1-p} +\abs{\mu_1'+\mu_2'}^{1-p} \right) > 0.
\]
Then, by \eqref{eq:A=0}, and also recalling that $\mathcal{E}=\mathcal{E}_1 \cup \mathcal{E}_2$,  we conclude that 
\[ \abs{(\overline \sigma, \underline \Sigma)\setminus (\mathcal{E}_1 \cup \mathcal{E}_2)}= \abs{(\overline \sigma, \underline \Sigma)\setminus \mathcal{E}}=0. \] 
Now, the same change of variable we performed in \eqref{change var} yields 
\[ \int_{\{t: \, \overline \sigma < v_*(t) < \underline \Sigma\}} H^p((v_*)') = - \int_{(\overline \sigma, \underline \Sigma) \setminus \mathcal{E}} H^p\left(\frac{1}{\mu'(s)}\right) \mu'(s)=0. \]
Since by our contradiction hypothesis we have $\overline \sigma < \underline \Sigma$, and, as $H(t)=0$ if and only if $t=0$,  we conclude that $v_*$ is constant, thus $\sigma=\Sigma$, a contradiction.  
\end{proof}
\begin{remark}\label{rem:I}
Let us observe that all the results of this section can be easily adapted to the case of a function $u$ defined in arbitrary interval $I\subset \R$ with $u_*$ defined in the same interval.
\end{remark}

\subsection{Anisotropic Symmetrization}\label{sec:anissym}
We again consider $H$ of the form \eqref{h:1dim} and, recalling  that the polar function $H_0$ of $H$ is defined as 
\[  
H_0(x)=\sup_{t \in \R} \frac{ t x}{H(t)} ,
\]
we have
\[ H_0(x)=
\begin{cases}
\frac x a & \text{ if } x \ge 0\\
- \frac x b & \text{ if } x < 0. 
\end{cases} 
\]
Let $I:=[0,1]$ and $u: I \to [0,+\infty)$, we define 
\[ 
I^\star= \left \{ x \in \R : H_0(-x) < \frac{1}{a+b} \right \} = \left( -\frac{a}{a+b}, \frac{b}{a+b}\right)
\]
and $u^\star: I^\star \to [0, +\infty)$ as 
\[
\begin{split}
u^\star(x)  &= \sup \{ t \in \overline \R: \abs{ \{ y: u(y) >t \} } > H_0(-x)(a+b) \} 
\\ 
& =  \begin{cases}
\sup \{ t : \abs{ \{ y: u(y) >t \} } > \frac{a+b}{b} x \} & \text{ if } x \ge 0 
\\\\
\sup \{ t : \abs{ \{ y: u(y) >t \} } > -\frac{a+b}{a} x \} & \text{ if } x < 0,
\end{cases}
\end{split}
\]
and we will call $u^\star$ the anisotropic rearrangement of $u$ with respect to $H_0$. 
\begin{remark}
For any set $E\subset \R$,  $E^{\star}$ is the interval $(\omega_{1},\omega_{2})$
with the same measure of $E$ and such that $\omega_1= -\frac a b \omega_2$; namely the sub-level set of $H_{0}(-\cdot)$ with the same measure of $E$.
Any interval satisfying $E=E^\star$   will be called an anisotropic ball. 
\end{remark}
We now introduce the anisotropic Polya inequality useful in our context.

 \begin{proposition}\label{prop: polya anis}
Let $u \in W_0^{1, p}(I,[0,+\infty))$,  then 
\begin{equation}\label{polya anis}
 \int_I H^p(u')\ge \int_{I^{\star}} H^p((u^\star)'). \end{equation}
Moreover, assume that $u \in C^1(I, [0,+\infty))$, and that  the set $\{ u'(x)=0 \}$ is finite. Then,  equality  in \eqref{polya anis} holds  if and only if $u \!\left(x+\frac{a}{a+b}\right)=u^\star(x)$. 
 \end{proposition}

\begin{remark}
Anisotropic Polya inequalities, together with the study of the equality case,  have first been proved in \cite{AlvinoFLT} (see also \cite{FV}) for every dimension  assuming $H(t \xi)=|t|H(\xi)$. Unfortunately,  the function $H$ given by \eqref{h:1dim}  does not enjoy this property.
Generalizations of \eqref{polya anis} are provided in \cite{VanSchaft} except for the study of the equality case. 
As we treat a quite simple situation, we provide a direct approach to show both
the inequality and the characterization of the equality case suitable to our situation.
  
\end{remark}	

\begin{proof}
Let us first prove the inequality \eqref{polya anis}. 
We will assume $u$ is piecewise affine, the case $W^{1,p}$ follows by density. We will adapt \cite[Theorem 2.9]{Kawohl}. 

Let  $\left\{ a_1 \leqslant \dots \leqslant a_k\right\}$ be the values of $u$ at the non-differentiability points, set $a_0=0$ and let 
 \[ \begin{split}
D_i &= \left\{ x \in [0, 1]: a_i < u(x) < a_{i+1} \right\} \quad\qquad E_i = \left\{ x \in I^{\star}: a_i < u^{\star}(x) < a_{i+1} \right\}=E_{i}^{-}\cup E^{+}_{i}
\\
E_i^{-} &= \left\{ x \in\left[-\frac{a}{a+b}, 0\right]: a_i < u^{\star}(x) < a_{i+1} \right\},\quad
E_i^{+} = \left\{ x \in\left[0,\frac{b}{a+b}\right]: a_i < u^{\star}(x) < a_{i+1} \right\}.
\end{split}\]

Taking into account that $H(0)=0$, it is enough to prove that
\begin{equation}\label{eq:goal}
\int_{D_i} H^p(u') \ge \int_{E_i} H^p\left((u^\star)'\right), \qquad \text{for $i=0,\dots, k-1.$}
\end{equation}

Recalling the decomposition \eqref{eq:decomposition-Di} we can consider the monotone and differentiable functions $\rho_{j}:(a_{i},a_{i+1})\mapsto Y_{ij}$ as in the proof of Proposition \ref{polya monotona} such that \eqref{eq:infoRhoj} holds and we notice that, since $u(0)=u(1)=0$, 
\[
\sign\, u'(\rho_j)=(-1)^{j+1}. 
\]
 
By definition, $u^{\star }$ is strictly increasing in $E_i^-$ and 
strictly decreasing in $E_i^+$, so that we can define $\rho_\pm^\star :(a_i, a_{i+1}) \to E_i^\pm$ such that $\rho_-^\star(\lambda)$ is the unique negative value such that $u^\star(\rho^\star_-)=\lambda$, and $\rho_+^\star(\lambda)$ the unique positive value such that $u^\star(\rho^\star_+)=\lambda$. In addition we have
\[
(u^\star)'\left(\rho_-^\star(\lambda)\right) = \left((\rho^\star_-)'(\lambda) \right)^{-1} \text{ in } E_i^- \;,\;\;  (u^\star)'\left(\rho_+^\star(\lambda)\right) = \left((\rho^\star_+)' (\lambda) \right)^{-1} \text{ in } E_i^+. 
\]
As a consequence, it results
\begin{equation}\label{eq:signrho}
\begin{split}
\rho_-^\star(\lambda)= \frac{a}{a+b} \sum_{j=1}^N (-1)^{j+1} \rho_j(\lambda)
\quad&\implies\quad
(\rho^\star_-)'(\lambda) = \frac{a}{a+b} \sum_{j=1}^N \abs{ \rho_j' (\lambda)}
\\
 \rho_+^\star(\lambda)= \frac{b}{a+b} \sum_{j=1}^N (-1)^j \rho_j(\lambda)
\quad\quad&\implies\quad
( \rho^\star_+)'(\lambda)= - \frac{b}{a+b} \sum_{j=1}^N \abs{  \rho_j' (\lambda) }. 
\end{split}\end{equation}

Then,  since showing \eqref{eq:goal}  is equivalent to prove
\[ 
\sum_{j=1}^{N(i)} \int_{Y_{ij} } H^p(u') \ge \int_{E_i^+} H^p((u^\star)') + \int_{E_i^-} H^p((u^\star)'), 
\]
we can exploit a change of variable and obtain the following  inequality
\[
 \sum_{j=1}^{N(i)} \int_{a_i}^{a_{i+1}} H^p\left[\left(  \rho'_j  \right)^{-1} \right]\abs{  \rho'_j} 
 \ge 
 \int_{a_i}^{a_{i+1}}  H^p\left[\left( (\rho_-^\star)' \right)^{-1} \right] \abs{ \left(\rho_-^\star\right)'}+  
 \int_{a_i}^{a_{i+1}}  H^p\left[
 \left( ( \rho_+^\star)' \right)^{-1} \right] \abs{ ( \rho_+^\star)'},
 \]
which, in view of \eqref{eq:signrho},  is equivalent to
\begin{align*} 
 \sum_{j=1}^{N(i)} \int_{a_i}^{a_{i+1}}\alpha_{j} H^p\left[\left(  \rho'_j  \right)^{-1} \right] 
 \ge &
 \int_{a_i}^{a_{i+1}}  
 \frac{a}{a+b} H^p\left[ \left( \frac{a}{a+b} \sum_{j=1}^{N(i)}  
 \abs{ \rho'_j} \right)^{-1}\right]
 \\
 &+ \int_{a_i}^{a_{i+1}}  \frac{b}{a+b} H^p\left[ \left(-\frac{b}{a+b} \sum_{j=1}^{N(i)}  \abs{ \rho'_j} \right)^{-1} \right] ,
\end{align*}
 where
\begin{equation}\label{eq:alphaj}
\alpha_j= \abs{  \rho'_j} \left( \sum_{j=1}^{N(i)}  \abs{ \rho'_j} \right)^{-1}\hskip-6pt,
\quad \text{so that }   \sum_{j=1}^{N(i)} \alpha_{j}=1.
\end{equation}
Then, it is sufficient to prove that
\begin{equation}\label{eq:integrand}
\begin{split}
\sum_{j=1}^{N(i)} 
\alpha_j H^p\left[ \left( \rho_j'  \right)^{-1} \right]  
\ge&
  \frac{a}{a+b} H^p\left[
\left(\frac{a}{a+b}\sum_{j=1}^{N(i)} \abs{ \rho'_j }\right)^{-1} \right]
\\
&+  \frac{b}{a+b} 
H^p\left[\left(-\frac{b}{a+b}\sum_{j=1}^{N(i)}  \abs{ \rho_j'} \right)^{-1} \right], 
\end{split}\end{equation}
for  $\lambda \in (a_i, a_{i+1})$.

Keeping in mind \eqref{eq:alphaj}, the  convexity of the real function $t^p$ and 
\eqref{h:1dim}, we obtain
 \begin{equation}\label{convex polya} 
 \begin{split}
 \sum_{j=1}^{N(i)} \alpha_j H^p\left[ \left(  \rho'_j\right)^{-1} \right] 
 & \geqslant
 \left\{\sum_{j=1}^{N(i)} \alpha_j H\left[ \left(\rho'_j \right)^{-1} \right] \right\}^p
= \left[
\sum_{j \text{ even}}-\alpha_j  b\left( \rho'_j\right)^{-1}  + \sum_{j \text{ odd}} \alpha_j a\left( \rho'_j \right)^{-1}  \right]^p
 \\
  & = \left[ \frac {N(i)}2 (a+b) \right]^p \left( \sum_{j=1}^{N(i)}  \abs{\rho'_j} \right)^{-p},
 \end{split}
 \end{equation}
 where we have also taken into consideration  that $N(i)$ is even. 
 On the other hand, on the right hand side of \eqref{eq:integrand} we have
 \[ \begin{split}
 H^p\left[\left(-\frac{b}{a+b}\sum_{j=1}^{N(i)}  \abs{ \rho_j'} \right)^{-1} \right]
 &= 
 b^p \left( \frac{b}{a+b} \right)^{-p} \left( \sum_{j=1}^{N(i)} 
  \abs{ \rho'_j} \right)^{-p}, 
\\
 H^p\left[\left(\frac{a}{a+b} \sum_j \abs{ \rho'_j} \right)^{-1} \right]
 &= a^p \left( \frac{a}{a+b} \right)^{-p} \left( \sum_{j=1}^{N(i)}  
 \abs{ \rho'_j} \right)^{-p}. 
 \end{split}
 \]
 This, together with \eqref{convex polya}, implies that \eqref{eq:integrand} is equivalent to
\begin{equation}\label{final ineq polya}  
\left[ \frac {N(i)}2 (a+b) \right]^p \ge (a+b)^p 
\end{equation}
which is satisfied as $N(i) \ge 2$. 

Let us now study the equality case. We preliminary observe that the arguments above also work for a function $u \in C^1(I)$ such that the set $\{ u'(x)=0 \}$ is finite, once we choose $\{ a_{0}\leqslant a_1 \leqslant \dots \leqslant a_k \}$ the values of $u$ at these points.

 With this choice, the functions $\rho_j:(a_i, a_{i+1}) \mapsto Y_{ij}$ are well defined and differentiable, and, by continuity, we can also define  $\rho_j(a_i)$ for any  $j=1, \dots, N(i)$ and any $i=0,\dots, k-1$.

In particular \eqref{convex polya} and \eqref{final ineq polya} hold. 

As a consequence, the strict convexity of the real function $t^{p}$ implies that
equality holds in \eqref{polya anis} if only if $N(i)=2$ for any $i$ and equality holds in \eqref{convex polya}.  Then,  \eqref{h:1dim} yields
\[ 
a \left(  \rho'_1 \right)^{-1}= H\left(\left(  \rho'_1  \right)^{-1} \right)=H\left[\left( \rho'_2  \right)^{-1} \right]=-b \left( \rho'_2 \right)^{-1}
\qquad \text{ in $(a_i, a_{i+1})$. }
\]
Namely,
\begin{equation}\label{eq rho} 
b \rho'_1(\lambda)=-a  \rho'_2(\lambda) \quad \text{ for all } \lambda \in (a_i, a_{i+1}), \, i=0, \dots, k-1. 
\end{equation}
Integrating this expression in $(0, t)$, $t \leqslant a_1$, we get 
 \[ \rho_2(\lambda)= -\frac{b}a \rho_1(\lambda) + \frac b a \rho_1(0) +\rho_2(0)= -\frac{b}a \rho_1(\lambda)+1, \]
and in particular
\[ \rho_2(a_1)=-\frac{b}a \rho_1(a_1)+1. \]
Repeating the argument on each interval $(a_i, a_{i+1})$, we conclude
\[ \rho_2(\lambda)=-\frac{b}a \rho_1(\lambda)+1,\qquad \text{for any $\lambda\in [0, \norm{u}_\infty)$.}
 \]
 Since $\rho_1(\norm{u}_\infty)=\rho_2(\norm{u}_\infty)$, we also deduce that $u$ attains its maximum for  $x=\frac{a}{a+b}$. 

We claim that superlevel sets of $u \!\left(x+\frac{a}{a+b}\right)$ are  anisotropic ball. Indeed, fix $t \in (0, \norm{u}_\infty)$, and let $x_1 \in [-\frac{a}{a+b}, 0]$ such that $u(x_1+\frac{a}{a+b})=t$, and $x_2 \in[0, \frac{b}{a+b}]$ such that $u(x_2+\frac{a}{a+b})=t$. Thus, 
\[ x_1+\frac{a}{a+b} = \rho_1(\lambda)= -\frac a b \rho_2(\lambda) + \frac a b = -\frac a b\left (x_2 + \frac{a}{a+b}\right) + \frac a b, \]
which immediately  implies $ x_1= -\frac a b x_2$.
Namely, the superlevel set of $u \!\left(x+\frac{a}{a+b}\right)$ at $t$
 is an anisotropic ball.  Hence
$u \!\left(x+\frac{a}{a+b}\right)=u^\star(x)$. 
\end{proof} 
\begin{remark}\label{rem:coarea}
We believe that the proof of the rigidity result in Proposition \ref{prop: polya anis}
may be addressed  as done in \cite{FV, ET}, paying attention to the effect produced by the loss of
evenness on the  anisotropy.  Indeed, one may
argue as in \cite{FV} to prove that equality  in  \eqref{polya anis} implies that the 
super-level sets of $u$ are intervals; then, under the additional assumption $|\{u'(x)=0\}|=0$,
one may show, following \cite{ET}, that they are centered in the same point.
Here, we have given a elementary proof suitable for our context.
\end{remark}

We end this section with the following analog of Proposition \ref{prop: polya anis} for negative functions. Recall that $\widetilde H(x):=H(-x)$, and that its polar 
 function  is given in \eqref{eq:H-zero-tilde}.  
\begin{corollary}\label{lem:anis polya 2}
Let $v: I:=[0, 1] \to (-\infty,0]$ such that $v \in W_{0}^{1,p}(0,1)$. 
 Define $v^\# :=-(-v)^\star:\widetilde{I}^\star \to (-\infty,0]$, where $(\cdot)^\star$ is the anisotropic rearrangement with respect to $\widetilde H_0(x)$, namely
\[  (-v)^\star(x)= 
\begin{cases}
\sup \{ t : \abs{ \{ y: -v(y) >t \} } \ge \frac{a+b}{a} x \} & \text{ if } x \ge 0 \\
\sup \{ t : \abs{ \{ y: -v(y) >t \} } \ge -\frac{a+b}{b} x \} & \text{ if } x < 0,
\end{cases} \]
with $\widetilde{I}^\star=(-\frac{b}{a+b}, \frac{a}{a+b})$. 
Then 
\begin{equation}\label{eq:polya neg} 
\int_I H^p(v') \ge \int_{\widetilde{I}^\star} H^p((v^\#)'). \end{equation}
Moreover, if $v \in C^1(I)$ and  the set $\{x : v'(x)=0 \}$ is finite, then equality holds in \eqref{eq:polya neg} if and only if $v\left(x+\frac{b}{a+b} \right)= v^\#(x)$.
\end{corollary}
\begin{proof}
Just notice that by Proposition \ref{prop: polya anis} 
\[ 
 \int_I H^p(v') = \int_I \widetilde H^p(-v') \geqslant
\int_{\widetilde{I}^\star} \widetilde H^p((-v^\star)')
= \int_{\widetilde{I}^\star} H^p((v^\#)').  
 \]
The statement about the equality case follows again by Proposition \ref{prop: polya anis}. 
\end{proof}

\section*{Acknowledgments} 
Work partially supported by  PRIN-2017-JPCAPN Grant: ``Equazioni 
differenziali alle derivate parziali non lineari'', by project Vain-Hopes within the program VALERE: VAnviteLli pEr la RicErca , by the INdAM-GNAMPA group,  by the Portuguese government through FCT - Funda\c c\~ao para a Ci\^encia e a Tecnologia, I.P., under the projects UID/MAT/04459/2020 and PTDC/MAT-PUR/1788/2020 and, when eligible, by COMPETE 2020 FEDER funds, under the Scientific Employment Stimulus - Individual Call (CEEC Individual) -\\ 2020.02540.CEECIND/CP1587/CT0008.

\end{document}